\documentclass{article} 
\usepackage[usenames,dvipsnames]{color}
\usepackage{amsmath} 
\usepackage{amssymb} 
\usepackage{bm}
\usepackage{amsthm}
\usepackage{CJK}
\usepackage{indentfirst}
\usepackage{tikz}
\usepackage{amsmath,amsfonts,amsthm,mathrsfs}
\usepackage{hyperref}
\numberwithin{equation}{section}
\usepackage{geometry}
\usepackage{extarrows}
\usepackage{setspace}

\numberwithin{equation}{section}
\newtheorem{theorem}{Theorem}[section]
\newtheorem{lemma}{Lemma}[section]

\newtheorem{corollary}{Corollary}[section]
\newtheorem{definition}{Definition}[section]
\newtheorem{remark}{Remark}[section]

\usepackage{algorithm,algpseudocode}
\usepackage{amsthm}
\usepackage{esint} 
\usepackage{subfigure}
\usepackage{lineno}

\newcommand{\relu}{\mbox{ReLU}}
\newcommand{\reluk}{\(\relu^k\)}
\def\NN{\mathbb{N}}
\def\RR{\mathbb{R}}
\def\BB{\mathbb{B}}

\def\SS{\mathbb{S}}
\def\YY{\mathbb{Y}}

\def\mal{\max\limits}

\def\sul{\sum\limits}

\def\lt{\left}
\def\rt{\right}
\newcommand{\lrt}[1]{\left({#1}\right)}

\def\mH{\mathcal{H}^\rr(\SS^d)}

\def\rr{ r}
\def\ss{ s}

\makeatletter
\renewcommand\@fnsymbol[1]{\arabic{footnote}}
\makeatother

\title{Integral Representations of Sobolev Spaces via ReLU$^k$ Activation Function and Optimal Error Estimates for Linearized Networks}

\author{ Xinliang Liu\thanks{King Abdullah University of Science and Technology, Thuwal 23955, Saudi Arabia} \and
  Tong Mao\footnotemark[1] \and
  Jinchao Xu\footnotemark[1]
}
\date{}

\begin{document}
	
\maketitle
\begin{abstract}

This paper presents two main theoretical results concerning shallow neural networks with ReLU$^k$ activation functions. We establish a novel integral representation for Sobolev spaces, showing that every function in $\mathcal{H}^{\frac{d+2k+1}{2}}(\Omega)$ can be expressed as an $\mathcal{L}^2$-weighted integral of ReLU$^k$ ridge functions over the unit sphere. This result mirrors the known representation of Barron spaces and highlights a fundamental connection between Sobolev regularity and neural network representations. Moreover, we prove that linearized shallow networks—constructed by fixed inner parameters and optimizing only the linear coefficients—achieve optimal approximation rates $\mathcal{O}(n^{-\frac{1}{2}-\frac{2k+1}{2d}})$ in Sobolev spaces.


 \end{abstract}

\section{Introduction}\label{sec:intro}

Deep neural networks (DNNs) are at the core of numerous breakthroughs in artificial intelligence \cite{LeCun2015}, providing a powerful framework for solving complex problems. The efficiency of deep learning depends on several critical factors, including expressive power, generalization ability, training efficiency, and robustness. 

Mathematically, a deep neural network is a specialized class of functions constructed through an iterative composition of shallow neural networks. Thus, the following shallow neural networks function class serves as a fundamental building block of artificial intelligence technology:
\begin{equation}\label{shallow}
    \Sigma_{n}^\sigma= \left\{
        \sum_{j=1}^n a_j \sigma\big(w_j\cdot x+b_j\big) : ~ w_j \in \mathbb{R}^{d},~a_j,b_j\in\RR
    \right\},
\end{equation}
Among various activation functions, the Rectified Linear Unit (ReLU) has emerged as the dominant choice in modern deep learning.  In this paper, we focus on ReLU  activation function, \(\relu(x) = \max(0, x)\), and its variants, \(\sigma_k (x)= \relu^k(x)\) ($k$ is a nonnegative integer). Due to the homogeneity of $\sigma_k$, we assume the parameters lie on the unit sphere $\mathbb{S}^d$. The corresponding function class of shallow ReLU$^k$ neural networks, denoted by $\Sigma_n^k = \Sigma_{n}^{\sigma_k}$, is then
\begin{equation}\label{shallowk}
    \Sigma_{n}^k= \left\{
        \sum_{j=1}^n a_j \sigma_k\big(\theta_j\cdot \tilde{x}\big) : ~ \theta_j \in \mathbb{S}^{d},~a_j\in\RR
    \right\},
\end{equation}
where $\displaystyle\tilde x=\binom{x}{1}$ and $\displaystyle\theta_j=\binom{w_j}{b_j}$. To
ensure encodability and numerical stability, it is essential to impose restrictions on the size
of the parameters in a neural network. Consequently, given a parameter $M$, which essentially
determines the complexity of the function class, the stable shallow neural network function class is defined as
\begin{equation}\label{shallowkM}
    \Sigma_{n,M}^k= \left\{
        \sum_{j=1}^n a_j \sigma_k\big(\theta_j\cdot \tilde{x}\big) : ~ \theta_j \in \mathbb{S}^{d},~\sul_{j=1}^n|a_j|\leq M
    \right\}.
\end{equation}

In the early works, qualitative convergence has been extensively investigated since the 1990s (see, e.g., \cite{cybenko1989approximation,hornik1989multilayer}). A fundamental result states that the universal approximation property holds, namely the functions in $\Sigma_n^k$ can approximate any continuous function, as long as the activation function \(\sigma\) is not a polynomial; see, e.g., \cite{leshno1993multilayer}.

The approximation rate of the shallow neural networks has also been widely studied in the literature. Roughly speaking, these results pertain to three types of function spaces: Barron spaces $\mathcal{B}^\sigma(\Omega)$, also known as variation spaces $\mathcal{K}_1(\mathbb D_\sigma)$ (see \eqref{eqn:barron_norm_closure} below for definition); spectral Barron space \cite{klusowski2018approximation,siegel2022high}; and Sobolev space (see \eqref{eqn:def_sob} below for definition).
In most existing works, the following error estimate has been established for the aforementioned function spaces:
\begin{equation}\label{half}
  \inf_{f_n\in \Sigma_n^\sigma}
  \|f - f_n\|_{\mathcal{L}^2(\Omega)}
  =\mathcal{O}(n^{-\frac{1}{2}}). 
\end{equation}
A key technique used to establish \eqref{half} can be traced back to \cite{pisier1981remarques}, which applied Maurey’s sampling method from \cite{maurey1973type}; see also \cite{Jones1992,barron1993universal,barron1994approximation,devore1996some,makovoz1998uniform,kurkova1,kurkova2,lewicki2004approximation,temlyakov2008greedy,Barron2008,E2019population,E2019barron,E2020representation}. An overview of aforementioned and other related results  can be found in \cite{DeVore1998,DeVore2021,konyagin2018some}.

Several studies have also explored improved approximation rates compared to \eqref{half}, namely:
\begin{equation}\label{halfmore}
  \inf_{f_n\in \Sigma_n^\sigma}
  \|f - f_n\|_{\mathcal{L}^2(\Omega)}
  = \mathcal{O}(n^{-\frac{1}{2}-\frac{\alpha}{d}})
\end{equation}
for some $\alpha>0$. For Barron spaces, we refer to  \cite{makovoz1996random,bach2017breaking,klusowski2018approximation,xu2020finite,siegel2022sharp,siegel2022optimal,mhaskar2023tractability}. For Sobolev spaces, we refer to \cite{mhaskar1993approximation,petrushev1998approximation,pinkus1999approximation}
.

Notably, \cite{klusowski2018approximation} and \cite{xu2020finite} established an approximation rate of $\mathcal{O}(n^{-\frac{1}{2}-\frac{k}{d}})$ for spectral Barron spaces $\tilde{\mathcal{B}}^{k}(\Omega)$, with further generalizations in \cite{ma2022uniform}. For closer study of $\tilde{\mathcal{B}}^{k}(\Omega)$, see also \cite{meng2022new}.

In the context of shallow ReLU$^k$ networks, a strong connection exists between Barron spaces $\mathcal{B}^k(\Omega) = \mathcal{B}^{\sigma_k}(\Omega)$ (see \eqref{eqn:barron_integ_def}) and the class of shallow networks \eqref{shallowkM} (for further details, see Theorem \ref{thm:SiegelXu}):
\begin{enumerate}
    \item A function $f$ can be approximated by neural network classes $\{\Sigma_{n,M}^k\}_{n=1}^\infty$ if and only if $f \in \mathcal{B}^k(\Omega)$ \cite{siegel2023characterization}.
    \item The Barron space $\mathcal{B}^k(\Omega)$ has an integration form \cite{siegel2023characterization}
    \begin{equation}
         \mathcal{B}^k(\Omega)=\lt\{\int_{\SS^d}\sigma_k(\theta\cdot\tilde x)d\mu(\theta):~\mu\in\mathcal{M}(\SS^d)\rt\}.
    \end{equation}
        Moreover,
    \begin{equation}
        \|f\|_{\mathcal{B}^k(\Omega)}\simeq\inf\limits_{\mu\in\mathcal{M}(\SS^d)}\lt\{|\mu|(\SS^d):~f(x)=\int_{\SS^d}\sigma_k(\theta\cdot\tilde x)d\mu(\theta)\rt\}.
    \end{equation}
    \item Shallow ReLU$^k$ networks achieve the optimal approximation rate in $\mathcal{B}^k(\Omega)$ \cite{siegel2022sharp}:
\begin{equation}\label{nonlinearfn}
    \inf_{f_n\in \Sigma_{n,M}^k}\|f-f_n\|_{\mathcal{L}^2(\Omega)}=\mathcal{O} (n^{-\frac{1}{2}-\frac{2k+1}{2d}}),
\end{equation}
where $M\simeq\|f\|_{\mathcal{B}^{k}(\Omega)}$.
\end{enumerate}
The constraint on the parameter bound $M$ ensures both encodability and numerical stability, which is crucial in the generalization analysis of machine learning models.

In this paper, we further investigate the approximation properties of shallow ReLU$^k$ networks. Specifically, we focus on a linear subset of $\Sigma_n^k$, which we refer to as the finite neuron space (FNS). Given a predetermined set of parameters $\displaystyle\left\{\theta_j^*\right\}_{j=1}^n \subset \SS^d$, we define the corresponding basis functions
\begin{equation*}
\phi_j(x) = \sigma_k(\theta_j^* \cdot \tilde{x}) = \sigma_k(w_j^* \cdot x + b_j^*), \quad j = 1, \dots, n,
\end{equation*}
and the finite neuron space as
\begin{equation}\label{linearV}
L_n^k= L_n^k(\left\{\theta_j^*\right\}_{j=1}^n) = \mathrm{span} \left\{ \phi_1, \dots, \phi_n \right\}.
\end{equation}
Analogous to $\Sigma_{n,M}^k$ in \eqref{shallowkM}, we define the constrained version of $L_n^k$ as
\begin{equation}
    L_{n,M}^k=\biggl\{\sul_{j=1}^na_j\phi_j:~\Bigl(n\sul_{j=1}^na_j^2\Bigr)^{\frac{1}{2}}\leq M\biggr\}.
\end{equation}
By applying the Cauchy–Schwarz inequality, it follows that
\begin{equation}\label{eqn:lin_subs_nonlin}
    L_{n,M}^k\subset\Sigma_{n,M}^k.
\end{equation}
A key advantage of the linear spaces $L_n^k$ and $L_{n,M}^k$ lies in their structural simplicity, which enables more efficient analysis and computation. Notably, optimization over $L_{n,M}^k$ reduces to a convex least squares problem, whereas training in $\Sigma_{n,M}^k$ generally involves solving a highly nonconvex optimization problem.

Similar to the results on ReLU$^k$ neural networks $\Sigma_{n,M}^k$, this paper presents a novel perspective on characterizing the finite neuron space $L_n^k$ and its constrained counterpart $L_{n,M}^k$. Under the assumption that the parameters $\{\theta_j^*\}_{j=1}^n$ form a well-distributed mesh over $\mathbb{S}^d$ (see the precise definition in \eqref{eqn:welldistr}), we establish the following results:
 \begin{enumerate}
    \item A function $f$ can be approximated by FNS $\{L_{n,M}^k\}_{n=1}^\infty$ if and only if $f\in\mathcal{H}^{\frac{d+2k+1}{2}}(\Omega)$;
    \item The Sobolev space $\mathcal{H}^{\frac{d+2k+1}{2}}(\Omega)$ has an integration form
    \begin{equation}
    \mathcal{H}^{\frac{d+2k+1}{2}}(\Omega)=\lt\{\fint_{\SS^d}\sigma_k(\theta\cdot\tilde x)\psi(\theta)d\theta:~\psi\in\mathcal{L}^2(\SS^d)\rt\}.
\end{equation}
Moreover,
\begin{equation}\label{eqn:int_rep_sob_intr}
    \|f\|_{\mathcal{H}^{\frac{d+2k+1}{2}}(\Omega)}\simeq\inf\limits_{\psi\in\mathcal{L}^2(\SS^d)}\Bigl\{\|\psi\|_{\mathcal{L}^2(\SS^d)}: f(x)=\fint_{\SS^d}\sigma_k(\theta\cdot\tilde x)\psi(\theta)d\theta\Bigr\},
\end{equation}
    \item The FNS achieves the optimal approximation rate in $\mathcal{H}^{\frac{d+2k+1}{2}}(\Omega)$:
\begin{equation}\label{linearfn}
    \inf_{f_n\in L_{n,M}^k}\|f-f_n\|_{\mathcal{L}^2(\Omega)}=\mathcal{O} (n^{-\frac{1}{2}-\frac{2k+1}{2d}}),
\end{equation}
where $M\simeq\|f\|_{\mathcal{H}^{\frac{d+2k+1}{2}}(\Omega)}$.
\end{enumerate}

We would like to note that the approximation estimate in \eqref{linearfn} is related to an earlier result by Petrushev \cite{petrushev1998approximation}. For the ReLU$^k$ activation function, the result in \cite{petrushev1998approximation} yields to an approximation rate estimate similar to \eqref{linearfn} by using a different construction of the weights. Specifically, it uses the tensor product points $\displaystyle\left\{\binom{w_j}{b_i}\right\}_{\substack{1\leq j\leq n_1\\1\leq i\leq n_2}}$, where $\{w_j\}_{j=1}^{n_1}$ 
are suitable quadrature points on $\SS^{d-1}$ and $\{b_i\}_{i=1}^{n_2}$ are well-distributed points on $[-1,1]$. As demonstrated in Corollary \ref{cor:petru} below, our result can be used to directly derive the  result in \cite{petrushev1998approximation} for the ReLU$^k$ case. However, it is important to note that our result cannot be derived from those in \cite{petrushev1998approximation}, nor can they be obtained using the proof techniques in \cite{petrushev1998approximation}. In fact, the proof in our paper differs significantly from that used in \cite{petrushev1998approximation}.

Most importantly, our analysis yields the following key coefficient estimate
    \begin{equation}\label{eqn:esti_coef_a}
    \sqrt{n}\|a\|_2\lesssim\|f\|_{\mathcal{H}^{\frac{d+2k+1}{2}}(\Omega)}.
\end{equation}
As studied in \cite{mao2025neural}, the bound on the coefficients such as \eqref{eqn:esti_coef_a} is essential in approximation theory in general. For example, it plays a central role in the generalization analysis (see Section \ref{sec:gene_analy}). More importantly, without a bound like \eqref{eqn:esti_coef_a}, approximation error estimates may lose their practical relevance \cite{mao2025neural}. In fact, \cite{lu2021deep} (see also \cite{mao2025neural}) showed that deep neural networks with a fixed number of parameters can achieve universal approximation when the weights are unbounded. Therefore, a bound estimate for coefficients is indispensable for both theoretical and practical applications.

Given the continuous embedding $\mathcal{H}^{\frac{d+2k+1}{2}}(\Omega) \hookrightarrow \mathcal{B}^k(\Omega)$ \cite{mao2024approximation}, the metric entropy of $\mathcal{H}^{\frac{d+2k+1}{2}}(\Omega)$ is asymptotically same as that of $\mathcal{B}^k(\Omega)$ (see \eqref{eqn:same_entropy}). Thus, linearized shallow networks achieve approximation rates comparable to the nonlinear class $\Sigma_{n,M}^k$. Comparing this characterization with $\mathcal{B}^k(\Omega)$ clarifies which functions are approximable by $\Sigma_{n,M}^k$ but not by $L_{n,M}^k$.

A consequence of \eqref{eqn:int_rep_sob_intr} is that for $f \in \mathcal{H}^{\frac{d+2k+1}{2}}(\Omega)$, there exists $\psi \in \mathcal{L}^2(\SS^d)$ (not necessarily unique) such that
\begin{equation}\label{eqn:f_expect_repre}
f(x)=\mathbb{E}\bigl[\psi(\theta)\sigma_k\bigl(\theta\cdot\tilde x\bigr) \bigr],
\end{equation}
where the expectation is over the uniform distribution on $\SS^d$. This identity provides a theoretical foundation for estimating approximation error in neural networks with random parameters—also known as random feature methods (see Section \ref{sec:random}). In particular, if $\{\theta_j\}_{j=1}^n$ are i.i.d. uniform samples on $\SS^d$, then
    \begin{equation}
        \mathbb{E}_n\Bigl[\inf\limits_{a\in\RR^n}\Bigl\|f(x)-\sul_{j=1}^na_j\sigma_k(\theta_j\cdot\tilde x) \Bigr\|_{\mathcal{L}^2(\Omega)}\Bigr]\lesssim n^{-\frac12}\|f\|_{\mathcal{H}^{\frac{d+2k+1}{2}}(\Omega)}.
    \end{equation}
    Our refined rate \eqref{linearfn} improves this bound to $\mathcal{O}\Big(\big(\frac{n}{\log n}\big)^{-\frac{1}{2}-\frac{2k+1} {2d}}\Big)$ via intricate analysis.
    
These insights also help explain the success of randomized approaches, which fix the inner weights ${\theta_j}$ independently of the data. This framework includes stochastic basis selection \cite{pao1994learning,Igelnik1995}, extreme learning machines \cite{Huang2006, huang2006universal}, random feature methods \cite{rahimi2007random, rahimi2008weighted, rahimi2008uniform}, and randomized neural networks \cite{saxe2011random, giryes2016deep}. Their theoretical underpinnings have been further studied in \cite{bach2017equivalence, li2019towards, gerace2020generalisation, hu2022universality, mei2022generalization}. Recent applications to numerical PDEs appear in \cite{dwivedi2020physics, dong2021local, nelsen2021random, chen2022bridging, dang2024local, wang2024extreme, chi2024random, zhang2024transferable}. Our results indicate that the effectiveness of these methods is primarily due to the randomly generated weights are nearly well-distributed (see \eqref{eqn:nonlinear_lower}), rather than any intrinsic randomness.

As another application of \eqref{linearfn}, we estimate the generalization error for learning $f \in \mathcal{H}^{\frac{d+2k+1}{2}}(\Omega)$. Let $\{(x_i, f(x_i))\}_{i=1}^m$ be i.i.d. uniform samples from $\Omega$. Under the setting of Theorem \ref{thm:generalization}, the bound holds:
\begin{equation}\label{eqn:gener_err}
    \mathbb{E}_{x_1,\dots,x_m}\Big[\|f_{n,m} - f\|_{\mathcal H^1(\Omega)}^2\Big]\lesssim (\|h\|_{\mathcal{L}^\infty(\Omega)}+M)Mm^{-\frac{1}{2}}.
\end{equation}
Here $f_{n,m}$ is the empirical risk minimizer:
\begin{equation}
    f_{n,m}:=\arg\min\limits_{f_n\in L_{n,M}^k}\frac{1}{m}\sul_{i=1}^m(f(x_i)-f_n(x_i))^2.
\end{equation}

Comparing with classical finite element space $V_n^k$ (piecewise polynomials of degree $\le k$ on $n$ elements \cite{Ciarlet1978}), our results show an advantage. Given the lower bound $\mathcal O(n^{-\frac{k+1}{d}})$ for finite elements \cite{lin2014lower}, there exists $f\in\mathcal{H}^{\frac{d+2k+1}{2}}(\Omega)$ such that
\begin{equation}
    \inf_{f_n\in L_n^k}\|f-f_n\|_{\mathcal{L}^2(\Omega)}=\mathcal{O} \left(n^{-\frac{1}{2}(1-\frac{1}{d})}\right)\inf_{f_n\in V_n^k}\|f-f_n\|_{\mathcal{L}^2(\Omega)}.
\end{equation}
This highlights that for smooth functions, $V_n^k$ suffers from the curse of dimensionality, while $L_n^k$ does not, despite both using piecewise polynomials of degree $k$.

We also prove error estimates similar to \eqref{linearfn} for lower-order Sobolev spaces:
\begin{equation}\label{linearSobolev}
    \inf_{f_n\in L_n^k}\|f-f_n\|_{\mathcal{L}^2(\Omega)}=\mathcal{O}(n^{-\frac{\rr}{d}})\|f\|_{\mathcal{H}^\rr(\Omega)},\qquad r\leq\frac{d+2k+1}{2}.
\end{equation}
The rate $\mathcal{O}(n^{-\frac{\rr}{d}})$ matches optimal rates for Sobolev spaces by classic methods like polynomials, Fourier series \cite{devore1993constructive}, and finite elements \cite{Ciarlet1978}. Early works \cite{pinkus1999approximation} showed this rate for shallow networks with smooth sigmoidal activations $\sigma$ (see also \cite{mhaskar1993approximation,petrushev1998approximation}).

Comparing with the $\mathcal{L}^2$-rate for shallow ReLU$^k$ networks \cite{mao2024approximation} (see also \cite{mao2023rates,yang2024optimal} for $\mathcal{L}^\infty$-norm)
\begin{equation}\label{nonlinearSobolev}
  \inf_{f_n\in \Sigma_n^k}
  \|f - f_n\|_{\mathcal{L}^2(\Omega)}
  \;=\mathcal{O}(
  n^{-\frac{\rr}{d}})\;\|f\|_{\mathcal{H}^\rr(\Omega)},
  \quad \rr\leq\frac{d+2k+1}{2},
\end{equation}
we conclude that optimal nonlinear estimates \eqref{nonlinearSobolev} are fully realized by linear approximation in Sobolev spaces. Rate \eqref{nonlinearSobolev} is also optimal up to log factors (Theorem \ref{thm:random_linear}).

This paper is organized as follows. Section \ref{sec:notation&main} introduces notation and main results. Section \ref{sec:harmonic_legendre} reviews spherical harmonics and Legendre polynomials essential for the proofs. Section \ref{sec:pf_sphere} establishes approximation results on the sphere using linear ReLU$^k$ spaces with well-distributed weights. Section \ref{sec:proof_ball} extends these results to general domains $\Omega\subset\RR^d$, completing the main proofs. Section \ref{sec:random} connects linearized shallow networks and random feature methods, showing the latter's approximation rate follows from the former. Section \ref{sec:gene_analy} applies our approximation results and coefficient estimates to compute the generalization error for an elliptic PDE.


\section{Background and main results}\label{sec:notation&main}

Before presenting our main results, we introduce the notation used in this paper and provide brief discussions of main results in this paper. Some of these notations were previously mentioned in Section \ref{sec:intro}, where they were assumed to be familiar to the reader. Here, we provide explicit definitions for clarity.

First of all, following \cite{xu1992iterative}, we adapt the notation $\gtrsim$, $\lesssim$, and $\simeq$, which express upper and lower bounds up to constant factors.  When we write
$$
f(x)\gtrsim g(x),\quad g(x)\lesssim h(x),\quad h(x)\simeq k(x)
$$
it means that there exist constants $c_1,c_2,c_3,c_4$ independent of $x$ such that
$$
f(x)\ge c_1 g(x),\quad g(x)\le c_2 h(x),\quad c_3 h(x)\le k(x)\le c_4 h(x).
$$


         

\subsection{Shallow neural networks, Barron spaces and Sobolev spaces}\label{subsec:barronspace}

Given an activation function $\sigma: \mathbb{R}\to \mathbb{R}$ and $G\subset\RR^{d+1}$, we denote the dictionary
$$\mathbb{D}_\sigma:=\left\{\pm\sigma(\theta\cdot\tilde x):~\theta\in G\right\},$$
and the $\mathcal{L}^2$-closure of its convex hull $B_1:=\overline{\mathrm{conv}(\mathbb{D}_\sigma)}$. Following \cite{siegel2022sharp} (see also \cite{Jones1992,devore1996some,kurkova1,kurkova2,lewicki2004approximation,temlyakov2008greedy,Barron2008}), we denote the Barron space $\mathcal{B}^\sigma(\Omega)$ as follows:
\begin{equation}\label{eqn:def_Barron}
    \mathcal{B}^\sigma(\Omega):=\{f\in\mathcal{L}^2(\Omega):~\|f\|_{\mathcal{B}^\sigma(\Omega)}<\infty\},
\end{equation}
where
\begin{equation*}
    \|f\|_{\mathcal{B}^\sigma(\Omega)}=\inf\left\{t>0:~f\in tB_1\right\}.
\end{equation*}

In particular, when $\sigma=\sigma_k$, we take $G=\SS^d$ as $\sigma_k$ is homogeneous. In this case, we denote
\begin{equation}\label{eqn:barron_norm_closure}
    \mathcal{B}^{k}(\Omega)=\mathcal{B}^{\sigma_k}(\Omega).
\end{equation}
The Barron spaces can be characterized by some qualitative approximation property of the following compact subset of the \reluk neural networks \eqref{shallowkM}. We state the relation between the \reluk neural networks and Barron spaces as the following theorem.

     \begin{theorem}[\cite{siegel2022sharp,siegel2023characterization}]\label{thm:SiegelXu}
     Let $d\in\NN$, $\Omega\subset\RR^d$ be a bounded domain.
     \begin{enumerate}
     \item $f\in\mathcal{B}^k(\Omega)$ if and only if, with some $M\simeq\|f\|_{\mathcal{B}^k(\Omega)}$,
     \begin{equation}\label{eqn:equiv_M_vari}
         \lim\limits_{n\to\infty}\inf\limits_{f_n\in\Sigma_{n,M}^k}\|f-f_n\|_{\mathcal{L}^2(\Omega)}=0.
     \end{equation}
    \item The Barron space $\mathcal{B}^k(\Omega)$ can be alternatively written as

     \begin{equation}\label{eqn:barron_integ_def}
         \mathcal{B}^k(\Omega)=\lt\{\int_{\SS^d}\sigma_k(\theta\cdot\tilde x)d\mu(\theta):~\mu\in\mathcal{M}(\SS^d)\rt\}.
    \end{equation}
    Furthermore,
    \begin{equation}\label{eqn:barron_norm_int}
        \|f\|_{\mathcal{B}^k(\Omega)}\simeq\inf\limits_{\mu\in\mathcal{M}(\SS^d)}\lt\{|\mu|(\SS^d):~f(x)=\int_{\SS^d}\sigma_k(\theta\cdot\tilde x)d\mu(\theta)\rt\}.
    \end{equation}
     \item Let $f\in\mathcal{B}^k(\Omega)$, with some $M\simeq\|f\|_{\mathcal{B}^k(\Omega)}$
         \begin{equation}\label{variation-space-approx-rate}
             \inf\limits_{f_n\in\Sigma_{n,M}^k}\lt\|f-f_n\rt\|_{\mathcal{L}^2(\Omega)} \lesssim n^{-\frac{1}{2}-\frac{2k+1}{2d}} \|f\|_{\mathcal{B}^k(\Omega)},
         \end{equation}
         All the corresponding constants are independent of $n$, $\{\theta_j^*\}_{j=1}^n$, and $f$.
     \end{enumerate}     
     \end{theorem}

Even if we remove the restriction on the coefficients, the rate \eqref{variation-space-approx-rate} is optimal up to logarithmic factors \cite{siegel2022sharp}:
\begin{equation}\label{eqn:nonlinear_lower}
    \sup\limits_{\|f\|_{\mathcal{B}^k(\Omega)}\leq1}\inf_{f_n\in \Sigma_n^k}\|f-f_n\|_{\mathcal{L}^2(\Omega)}\ge c \left(n\log n\right)^{-\frac{1}{2}-\frac{2k+1}{2d}}.
\end{equation}
We also refer to \cite{maiorov2000near,konyagin2018some,bartlett2019nearly,siegel2022optimal1} for relevant results. 

We now introduce the Sobolev spaces $\mathcal{H}^\rr(\Omega)$. Let $\Omega\subset\RR^d$ be bounded, for $\rr\in\NN$, the Sobolev spaces are defined as
\begin{equation}
    \mathcal{H}^\rr(\Omega)=\left\{f\in\mathcal{L}^2(\Omega):~\|f\|_{\mathcal{H}^{\rr}(\Omega)}<\infty\right\}.
\end{equation}
where norm is given by
\begin{equation}\label{eqn:def_sob}
    \|f\|_{\mathcal{H}^{\rr}(\Omega)}=\Bigl(\|f\|_{\mathcal{L}^2(\Omega)}^2+\sum\limits_{{\alpha}\in\NN_0^d,\|{\alpha}\|_1=\rr}\|D^{\alpha} f\|_{\mathcal{L}^2(\Omega)}^2\Bigr)^{\frac{1}{2}}.
\end{equation}
For $\rr\notin\NN$, $\mathcal{H}^\rr(\Omega)$ is defined through a standard interpolation (see, e.g., \cite{bergh2012interpolation}) of $\mathcal{H}^{\lfloor\rr\rfloor}(\Omega)$ and $\mathcal{H}^{\lceil\rr\rceil}(\Omega)$.

In this paper, we establish an approximation rate of $\mathcal{O}(n^{-\frac{\rr}{d}})$ for functions in the Sobolev space $\mathcal{H}^\rr(\Omega)$ under the condition that these parameters are well-distributed on $\SS^d$.
\begin{definition}[Well-distributed and quasi-uniform]
    Let $d\in\NN$, a set of points $\{\theta_j^*\}_{j=1}^n\subset\SS^d$ is said to be well-distributed if
    \begin{equation}\label{eqn:welldistr}
        \mal_{\theta\in\SS^d}\min\limits_{1\leq j\leq n}\rho(\theta,\theta_j^*)\lesssim n^{-\frac{1}{d}},
    \end{equation}
    Moreover, a well-distributed collection $\{\theta_j^*\}_{j=1}^n$ is said to be quasi-uniform if
    \begin{equation}\label{eqn:quasiuniform}
        \mal_{\theta\in\SS^d}\min\limits_{1\leq j\leq n}\rho(\theta,\theta_j^*)\lesssim\min\limits_{i\neq j}\rho(\theta_i^*,\theta_j^*).
    \end{equation}
    The corresponding constants are independent of $n$.
\end{definition}


\subsection{Main results}

We now present the main results of this paper. The first theorem focuses on analyzing the approximation properties of the ReLU$^k$ activation function with predetermined parameters $\{\theta_j^*\}_{j=1}^n$.

\begin{theorem}\label{thm:appr_rate_ball}
    Let $d,n\in\NN$, $k\in\NN_0$, $\Omega \subset \mathbb{R}^d$ be a bounded domain with Lipschitz boundary, and $\{\theta_j^*\}_{j=1}^n\subset\SS^d$. Then for any $f\in\mathcal{H}^{\frac{d+2k+1}{2}}(\Omega)$, with some     $$M\simeq\|f\|_{\mathcal{H}^{\frac{d+2k+1}{2}}(\Omega)},$$    \begin{equation}\label{eqn:main_res_nonuniform}
        \inf\limits_{f_n\in L_{n,M}^k }\bigl\|f-f_n\bigr\|_{\mathcal{L}^2(\Omega)}\lesssim h^{\frac{d+2k+1}{2}}\lt\|f\rt\|_{\mathcal{H}^{\frac{d+2k+1}{2}}(\Omega)}.
    \end{equation}
    where
    $$h=\mal_{\theta\in\SS^d}\min\limits_{1\leq j\leq n}\rho(\theta,\theta_j^*).$$
    More generally, let $\rr\leq\frac{d+2k+1}{2}$ and $\ss\leq\min\{k,\rr\}$. Then for any $f\in\mathcal{H}^\rr(\Omega)$, with some $\displaystyle M\simeq h^{-\frac{d+2k+1-2\rr}{2}}\|f\|_{\mathcal{H}^{\rr}(\Omega)},$
    \begin{equation}\label{eqn:main_Omega_beta}
        \inf\limits_{f_n\in L_{n,M}^k }\bigl\|f-f_n\bigr\|_{\mathcal{H}^\ss(\Omega)}\lesssim h^{\rr-\ss}\lt\|f\rt\|_{\mathcal{H}^\rr(\Omega)}.
    \end{equation}

    If the collection $\{\theta_j^*\}_{j=1}^n\subset\SS^d$ is well-distributed, then for any $f\in\mathcal{H}^{\frac{d+2k+1}{2}}(\Omega)$, with some
    $M\simeq\|f\|_{\mathcal{H}^{\frac{d+2k+1}{2}}(\Omega)}$,
    we have
    \begin{equation}\label{eqn:rate_d+2k+1}
        \inf\limits_{f_n\in L_{n,M}^k }\bigl\|f-f_n\bigr\|_{\mathcal{L}^2(\Omega)}\lesssim n^{-\frac{1}{2}-\frac{2k+1}{2d}}\lt\|f\rt\|_{\mathcal{H}^{\frac{d+2k+1}{2}}(\Omega)}.
    \end{equation}   
    All the corresponding constants are independent of $n$, $\{\theta_j^*\}_{j=1}^n$, and $f$.
\end{theorem}
\begin{remark}\label{rem:exclude_cap}
    The points $\{\theta_j^*\}_{j=1}^n$ do not necessarily need to be distributed over the entire domain $\SS^d$. By applying a shift, we may assume that $0 \in \Omega$. Let $\lambda_\Omega := \mathrm{diam}(\Omega)$. Then, for any $x \in \Omega$, we have
    \begin{eqnarray*}
    &\sigma_k(\theta\cdot\tilde x)=(\theta\cdot\tilde x)^k,\qquad&\theta\in G_+:=\biggl\{\theta\in\SS^d:~\theta_{d+1}\geq\frac{\lambda_\Omega}{\sqrt{1+\lambda_\Omega^2}}\biggr\},\\
            &\sigma_k(\theta\cdot\tilde x)=0,\qquad&\theta\in G_-:=\biggl\{\theta\in\SS^d:~\theta_{d+1}\leq-\frac{\lambda_\Omega}{\sqrt{1+\lambda_\Omega^2}}\biggr\}.
    \end{eqnarray*}
    We take points $\{\vartheta_j^*\}_{j=1}^{\binom{k+d}{d}}$ in $G_+$ such that the collection $\{(\vartheta_j^*\cdot\tilde x)^k\}_{j=1}^{\binom{k+d}{d}}$ are linearly independent. To get the linear approximation rates in Theorem \ref{thm:appr_rate_ball}, it suffices to combine the points $\{\vartheta_j^*\}_{j=1}^{\binom{k+d}{d}}$ with a collection of well-distributed points in
    \begin{equation*}
        \SS^d\setminus(G_+\cup G_-)=\biggl\{\theta\in\SS^d:~|\theta_{d+1}|\leq\frac{\lambda_\Omega}{\sqrt{1+\lambda_\Omega^2}}\biggr\}.
    \end{equation*}
    It is worth noting that $\{\vartheta_j^*\}_{j=1}^{\binom{k+d}{d}}$ can also be replaced by a collection of points $\{\theta_j^*\}_{j=1}^{\binom{k+d}{d}}\cup\{-\theta_j^*\}_{j=1}^{\binom{k+d}{d}}$ in $\SS^d\setminus(G_+\cup G_-)$, since there is the relation $t^k=\sigma_k(t)+(-1)^k\sigma_k(-t)$.
\end{remark}

\begin{corollary}[Theorem 8.2. \cite{petrushev1998approximation}]\label{cor:petru}
    Let $d,n,k,\Omega$ be as in Theorem \ref{thm:appr_rate_ball}, $n_1 \simeq n^{\frac{d-1}{d}}$ and $n_2 \simeq n^{\frac{1}{d}}$, there exists quadrature points $\{w_j^*\}_{j=1}^{n_1}\subset\SS^{d-1}$ and uniform distributed points $\{b_i^*\}_{i=1}^{n_2}\subset[-1,1]$, such that for any $f\in\mathcal{H}^{\frac{d+2k+1}{2}}(\BB^d)$,
    \begin{equation}\label{eqn:petrushev}
    \inf\limits_{a_{i,j}\in\RR,\forall i,j}\Big\|f(x)-\sul_{j=1}^{n_2}\sul_{i=1}^{n_1}a_{i,j}\sigma_k(w_j^*\cdot x+b_i^*)\Big\|_{\mathcal{L}^2(\BB^d)}\lesssim n^{-\frac{1}{2}-\frac{2k+1}{2d}}\lt\|f\rt\|_{\mathcal{H}^{\frac{d+2k+1}{2}}(\BB^d)}.
\end{equation}
\end{corollary}
\begin{proof}
The points constructed in \cite{petrushev1998approximation} satisfies 
$$\max\limits_{\binom{w}{b}\in\SS^{d-1}\times[-1,1]}\min\limits_{\substack{1\leq j\leq n_1\\1\leq i\leq n_2}}\Big|\binom{w}{b}-\binom{w_j}{b_i}\Big|\lesssim n^{-\frac{1}{d}}.$$
    Then it suffices to notice the normalization $\mathfrak{P}:\SS^{d-1}\times[-1,1]\to\{\theta\in\SS^{d}:~|\theta_{d+1}|\leq\frac{1}{\sqrt{2}}\}$
    \begin{equation}
        \mathfrak{P}:~\binom{w}{b}\mapsto\frac{1}{\sqrt{1+b^2}}\binom{w}{b},\qquad\binom{w}{b}\in\SS^{d-1}\times[1,1].
    \end{equation}
    is a uniform homeomorphism. Together with Remark \ref{rem:exclude_cap} completes the proof.
\end{proof}



As we mentioned in \eqref{eqn:nonlinear_lower}, Theorem \ref{thm:appr_rate_ball} is optimal up to logarithmic factors. Interestingly, it leads to an important function space embedding result, previously established in \cite{mao2024approximation}:
\begin{equation}\label{eqn:var_embedding}
    \mathcal{H}^{\frac{d+2k+1}{2}}(\Omega)\hookrightarrow\mathcal{B}^k(\Omega),
\end{equation}
and the gap between the above two spaces is small in the sense that the metric entropies of their unit balls are of the same order, namely
    \begin{equation}\label{eqn:same_entropy}
    \epsilon_n\bigl(\BB\bigl(\mathcal{H}^{\frac{d+2k+1}{2}}(\Omega)\bigr)\bigr)_{\mathcal{L}^2(\Omega)}\simeq\epsilon_n\lrt{\BB\lrt{\mathcal{B}^k(\Omega)}}_{\mathcal{L}^2(\Omega)}\simeq n^{-\frac{d+2k+1}{2d}}.
    \end{equation}
    Here the metric entropies mean the infimum of $\lambda$ such that $2^n$ $\mathcal{L}^2$-balls of radius $\lambda$ can cover the unit balls in $\mathcal{H}^{\frac{d+2k+1}{2}}(\Omega)$ and $\mathcal{B}^k(\Omega)$ respectively \cite{kolmogorov1958linear}. The first equivalence in \eqref{eqn:same_entropy} follows from classical results on Sobolev space (see, e.g., \cite{devore1993constructive}), and the second equivalence was recently established in 
     \cite{siegel2022sharp}. 
\begin{corollary}[Theorem 1. \cite{mao2024approximation}]\label{cor:embedding}
Let $k,d,\Omega$ be as in \eqref{eqn:welldistr}, then there is the continuous embedding
$$\mathcal{H}^{\frac{d+2k+1}{2}}(\Omega)\hookrightarrow\mathcal{B}^k(\Omega).$$
\end{corollary}

\begin{proof}
    Let $f\in \mathcal{H}^{\frac{d+2k+1}{2}}(\Omega)$ and $f_n\in L_{n,M}^k$ be the approximants of $f$ in Theorem \ref{thm:appr_rate_ball}. Then
    \begin{equation*}
        f_n\in L_{n,M}^k\subset\Sigma_{n,M}^k.
    \end{equation*}
    Notice that $\lim\limits_{n\to\infty}\|f_n-f\|_{\mathcal{L}^2(\Omega)}=0$ implies
    \begin{equation}
        f=\lim\limits_{n\to\infty}f_n\subset\overline{\bigcup\limits_{n=1}^\infty\Sigma_{n,M}^k}=M\overline{\mathrm{conv}(\mathbb{D}_{\sigma_k})}.
    \end{equation}
    It means
    \begin{equation*}
        \|f\|_{\mathcal{B}^k(\Omega)}\leq M\lesssim\|f\|_{\mathcal{H}^{\frac{d+2k+1}{2}}(\Omega)}.
    \end{equation*}
\end{proof}
The embedding theorem provides some aspect of characterizing the smoothness of Barron spaces in terms of smoothness. In particular, the metric entropy argument confirms that $\mathcal{H}^{\frac{d+2k+1}{2}}(\Omega)$ is the maximal Sobolev space for the embedding. We mentioned the integral representation for Barron spaces in \eqref{eqn:barron_integ_def}, which coincides with the definitions used in \cite{E2019population,E2019barron,E2020representation} (see also \cite{yukich1995sup,makovoz1998uniform} for similar spaces). To further clarify the difference between $\mathcal{B}^k(\Omega)$ and $\mathcal{H}^{\frac{d+2k+1}{2}}(\Omega)$, we present a new integral representation result for Sobolev spaces in comparison.

\begin{theorem}\label{thm:barron_sob_equ}
Let $d\in\NN$, $k\in\NN_0$, $\Omega\subset\RR^d$ be a bounded domain, then the Sobolev space $\mathcal{H}^{\frac{d+2k+1}{2}}(\Omega)$ given in \eqref{eqn:def_sob} can be alternatively written as
\begin{equation}\label{eqn:sob_int_def}
    \mathcal{H}^{\frac{d+2k+1}{2}}(\Omega)=\lt\{\fint_{\SS^d}\sigma_k(\theta\cdot\tilde x)\psi(\theta)d\theta:~\psi\in\mathcal{L}^2(\SS^d)\rt\}.
\end{equation}
In particular, for any $f\in \mathcal{H}^{\frac{d+2k+1}{2}}(\Omega)$, there exists a $\psi\in \mathcal{L}^2(\SS^d)$ (that may not be unique) such that
\begin{equation}\label{eqn:fpsi_int}
    f(x)=\fint_{\SS^d}\sigma_k(\theta\cdot\tilde x)\psi(\theta)d\theta.
\end{equation}
Furthermore, 
\begin{equation}
    \|f\|_{\mathcal{H}^{\frac{d+2k+1}{2}}(\Omega)}\simeq\inf\limits_{\psi\in\mathcal{L}^2(\SS^d)}\bigl\{\|\psi\|_{\mathcal{L}^2(\SS^d)}: \psi \mbox{ satisfies }\eqref{eqn:fpsi_int}\bigr\},
\end{equation}
\end{theorem}

This theorem provides a new characterization of Sobolev spaces and highlights the structural differences between Barron spaces and Sobolev spaces. To specify, \eqref{eqn:barron_integ_def} and Theorem \ref{thm:barron_sob_equ} establish a space isomorphism (as vector spaces, not Banach spaces)
\begin{equation*}
    \mathcal{B}^k(\Omega)/\mathcal{H}^{\frac{d+2k+1}{2}}(\Omega)\cong(\mathcal{M}(\SS^d)/\mathcal{N}_k)/(\mathcal{L}^2(\SS^d)/(\mathcal{L}^2(\SS^d)\cap\mathcal{N}_k))\cong\mathcal{M}(\SS^d)/(\mathcal{L}^2(\SS^d)+\mathcal{N}_k),
\end{equation*}
where $\mathcal{N}_k$ is the kernel of the homomorphism $\mu\mapsto\int_{\SS^d}\sigma_k(\theta\cdot\tilde x)d\mu(\theta)$:
\begin{equation}
    \mathcal{N}_k=\Bigl\{\mu\in\mathcal{M}(\SS^d):~\int_{\SS^d}\sigma_k(\theta\cdot\tilde x)d\mu(\theta)=0,\quad a.e.\Bigr\},
\end{equation}
and $\mathcal{M}(\SS^d)$ is the space of signed Borel measures. Notably, we slightly abuse notation by treating $\mathcal{L}^2(\mathbb{S}^d)$ as a subspace of $\mathcal{M}(\mathbb{S}^d)$. Specifically, any function $\psi \in \mathcal{L}^2(\mathbb{S}^d)$ can be naturally identified with the measure $\mu \in \mathcal{M}(\mathbb{S}^d)$ whose density function is given by $\psi$.

An interesting special case arises when considering even/odd Barron spaces and Sobolev spaces defined on the sphere $\mathbb{S}^d$. In this setting, $\mathcal{N}_k=\{0\}$, which implies that the quotient space $\mathcal{B}^k(\mathbb{S}^d)/\mathcal{H}^{\frac{d+2k+1}{2}}(\mathbb{S}^d)$ is almost equivalent to the quotient $\mathcal{M}(\mathbb{S}^d)/\mathcal{L}^2(\mathbb{S}^d)$, up to even/odd function spaces (see Theorem \ref{thm:sph_int_rep_sob}).

Moreover, this characterization has direct implications for function approximation. Specifically, it implies that functions that can be approximated as in \eqref{eqn:rate_d+2k+1} necessarily belong to $\mathcal{H}^{\frac{d+2k+1}{2}}(\Omega)$. We summarize all these insights above and provide an analogue of Theorem \ref{thm:SiegelXu}.

\begin{theorem}\label{cor:chara_sob}
    Let $d\in\NN$, $k\in\NN_0$, $\Omega\subset\RR^d$ be a bounded domain, then
     \begin{enumerate}
     \item $f\in\mathcal{H}^{\frac{d+2k+1}{2}}(\Omega)$ if and only if, with some $M\simeq\|f\|_{\mathcal{H}^{\frac{d+2k+1}{2}}(\Omega)}$, there exists a sequence of spaces $\{L_{n,M}^k\}_{n=1}^\infty$ induced by quasi-uniform weights $\{\theta_{j,n}^*\}_{j=1}^n$,
     \begin{equation}\label{eqn:appr_equiv_sob}
         \lim\limits_{n\to\infty}\inf\limits_{\substack{f_n\in L_{n,M}^k}}
         \bigl\|f-f_n\bigr\|_{\mathcal{L}^2(\Omega)}=0.
     \end{equation}
    \item The Sobolev space $\mathcal{H}^{\frac{d+2k+1}{2}}(\Omega)$ can be alternatively written as

     \begin{equation}
         \mathcal{H}^{\frac{d+2k+1}{2}}(\Omega)=\lt\{\int_{\SS^d}\sigma_k(\theta\cdot\tilde x)\psi(\theta)d\theta:~\psi\in\mathcal{L}^2(\SS^d)\rt\}.
    \end{equation}
    Furthermore,
    \begin{equation}
        \|f\|_{\mathcal{H}^{\frac{d+2k+1}{2}}(\Omega)}\simeq\inf\limits_{\psi\in\mathcal{L}^2(\SS^d)}\lt\{\|\psi\|_{\mathcal{L}^2(\SS^d)}:~f(x)=\int_{\SS^d}\sigma_k(\theta\cdot\tilde x)\psi(\theta)d\theta\rt\}.
    \end{equation}
     \item Let $f\in\mathcal{H}^{\frac{d+2k+1}{2}}(\Omega)$, with some $M\simeq\|f\|_{\mathcal{H}^{\frac{d+2k+1}{2}}(\Omega)}$, we have
         \begin{equation}
    \inf\limits_{\substack{f_n\in L_{n,M}^k}}
         \bigl\|f-f_n\bigr\|_{\mathcal{L}^2(\Omega)}\lesssim n^{-\frac{1}{2}-\frac{2k+1}{2d}} \|f\|_{\mathcal{H}^{\frac{d+2k+1}{2}}(\Omega)}.
         \end{equation}
         \end{enumerate}
    All the corresponding constants are independent of $n$, $\{\theta_j^*\}_{j=1}^n$, and $f$.
\end{theorem}

\begin{proof}
    Given Theorem \ref{thm:appr_rate_ball} and \ref{thm:barron_sob_equ}, we only need to prove \eqref{eqn:appr_equiv_sob} implies $f\in\mathcal{H}^{\frac{d+2k+1}{2}}(\Omega)$. Let each $f_n$ in \eqref{eqn:appr_equiv_sob} has the form 
    $$f_n(x)=\sul_{j=1}^na(n)_j\sigma_k(\theta_{j,n}^*\cdot\tilde x).$$
    We can take a disjoint quasi-uniform partition $\SS^d=\bigcup_{j=1}^nA_{j,n}$ with
    \begin{equation}
        \theta_{j,n}^*\in A_{j,n},\quad\fint_{A_{j,n}}1d\eta\simeq n^{-1},\quad\mathrm{diam}(A_{j,n})\lesssim n^{-\frac{1}{d}},\qquad j=1,\dots,n.
    \end{equation}
    Then the piecewise constant function $\psi_n$ given by $\displaystyle\psi_n=\sul_{j=1}^n\frac{a(n)_j}{\fint_{A_{j,n}}d\eta}\mathbf{1}_{A_{j,n}}$
    has a convergence subsequence. One could verify the limit of this subsequence is the desired function $\psi$.
\end{proof}

We now show that Sobolev Spaces can be characterized as reproducing kernel Hilbert spaces (RKHS) associated with the ReLU$^k$ activation function. A common way to define a Reproducing Kernel Hilbert Space (RKHS) is through a feature map. Let $\Omega$ be a set, $\mathcal{H}_0$ be a Hilbert space, and $\Phi_0: \Omega \to \mathcal{H}_0$ be a feature map. The space of functions
\begin{equation} \label{eq:rkhs_from_feature_map_def}
    \mathcal{H} := \left\{ f: \Omega \to \mathbb{R} \mid \exists w \in \mathcal{H}_0 \text{ such that } f(x) = \langle w, \Phi_0(x) \rangle_{\mathcal{H}_0} \text{ for all } x \in \Omega \right\}
\end{equation}
equipped with the norm
\begin{equation} \label{eq:rkhs_norm_inf_def_revised}
    \|f\|_{\mathcal{H}} := \inf \left\{ \|w\|_{\mathcal{H}_0} \mid w \in \mathcal{H}_0 \text{ and } f(\cdot) = \langle w, \Phi_0(\cdot) \rangle_{\mathcal{H}_0} \right\}
\end{equation}
is an RKHS. The reproducing kernel for this space is given by $K(x, y) = \langle \Phi_0(x), \Phi_0(y) \rangle_{\mathcal{H}_0}$ (see, e.g., \cite{awad2008support}).  We can apply this construction to the ReLU$^k$ activation function. This establishes a direct connection between the Sobolev space and the RKHS associated with the ReLU$^k$ activation function.
\begin{theorem}\label{thm:sobolev_is_rkhs}
    Let $K_{\text{ReLU}^k}(x,y) = \int_{\SS^d} \sigma_k(\theta \cdot \tilde{x}) \sigma_k(\theta \cdot \tilde{y}) d\theta$ be the kernel defined on $\Omega$, with feature map $\Phi_0: \Omega \to \mathcal{L}^2(\SS^d)$ given by $(\Phi_0(x))(\theta) = \sigma_k(\theta \cdot \tilde{x})$. The Sobolev space $\mathcal{H}^{\frac{d+2k+1}{2}}(\Omega)$ is the  reproducing kernel hilbert space (RKHS) associated with the kernel $K_{\text{ReLU}^k}(x,y)$. Its Sobolev norm $\|\cdot\|_{\mathcal{H}^{\frac{d+2k+1}{2}}(\Omega)}$ is equivalent to the canonical RKHS norm, $\|f\|_{\mathcal{H}_{K_{\text{ReLU}^k}}}$, associated with $K_{\text{ReLU}^k}$, where
    \begin{equation*}
        \|f\|_{\mathcal{H}_{K_{\text{ReLU}^k}}} = \inf\limits_{\psi\in\mathcal{L}^2(\SS^d)}\left\{\|\psi\|_{\mathcal{L}^2(\SS^d)}:~f(x)=\int_{\SS^d}\sigma_k(\theta\cdot\tilde x)\psi(\theta)d\theta\right\}.
    \end{equation*}
\end{theorem}

\begin{proof}
    The proof directly follows from the definition of the RKHS and the integral representation of Sobolev spaces in Theorem \ref{thm:barron_sob_equ}. 
\end{proof}

To understand the difference between nonlinear and linear \reluk networks, we can compare Theorem \ref{cor:chara_sob} with Theorem \ref{thm:SiegelXu}. 
We remark that the differences between the conditions here and those in Theorem \ref{thm:SiegelXu} are (1) the parameters $\{\theta_j\}_{j=1}^n$ here are well-distributed whereas they are arbitrary in Theorem \ref{thm:SiegelXu}; and (2) the condition $\|a\|_2\leq \frac{M}{\sqrt{n}}$ here is stronger than $\|a\|_1\leq M$ in Theorem \ref{thm:SiegelXu}, which follows from Schwartz's inequality.

\section{Spherical harmonics and Legendre polynomials}\label{sec:harmonic_legendre}

This subsection briefly reviews harmonic analysis on the unit sphere $\SS^d := \{\eta\in \RR^{d+1}:~|\eta| = 1\}$, following \cite{dai2013approximation,stein1971introduction}. The average integral uses the normalized surface measure $d\eta$, $$\fint_{\SS^d} f(\eta) \, d\eta=\frac{1}{\omega_d}\int_{\SS^d}f(\eta)d\eta,\qquad f\in\mathcal{L}^1(\SS^d),$$
where $\omega_d=\int_{\SS^d}1d\eta$. The geodesic distance is $\rho(\eta,\theta)=\arccos(\eta\cdot \theta)$.

Let $\mathbb{P}_m(\SS^d)$ be the space of polynomials of degree at most $m$ restricted to $\SS^d$, with inner product $$\langle p,q\rangle_{\mathcal{L}^2(\SS^d)} := \fint_{\SS^d} p(\eta)q(\eta)d\eta.$$
The dimension of $\mathbb{P}_m(\SS^d)$ is $\displaystyle\binom{d+1+m}{m}$ for $m=0,1$ and $\displaystyle\binom{d+1+m}{m} - \binom{d-1+m}{m-2}$ for $m \geq 2$.

The space of spherical harmonics $\YY_m$ is the orthogonal complement of $\mathbb{P}_{m-1}(\SS^d)$ in $\mathbb{P}_m(\SS^d)$, which is known as the space of spherical harmonics of degree $m$. Let $\{Y_{m,\ell}\}_{\ell=1}^{N(m)}$ be an orthonormal basis for $\YY_m$, then its dimension is $N(0)=\text{dim}(\YY_0)=1$ and
$$N(m) = \text{dim}(\YY_m)=\frac{2m+d-1}{m}\binom{m+d-2}{d-1},\qquad m\geq 0.$$
By Weierstrass’ theorem, any $f\in \mathcal{L}^2(\SS^d)$ has the harmonic expansion $$f(\eta)=\sul_{m=0}^\infty\sul_{\ell=1}^{N(m)}\widehat{f}(m,\ell)Y_{m,\ell}(\eta),\qquad a.e.~\eta\in\SS^d,$$
where $\widehat{f}(m,\ell)=\lt<f,Y_{m,\ell}\rt>_{\mathcal{L}^2(\SS^d)}$. The $\mathcal{L}^2$-projection onto $\YY_m$ is $\Pi_m f = \sul_{\ell=1}^{N(m)}\widehat{f}(m,\ell)Y_{m,\ell}$.

Given the definition of projections, we are ready to define the Sobolev spaces.

\begin{definition}[Sobolev spaces on the sphere]\label{def:sobolev_sphere}
    For $\rr>0$, the Sobolev space $\mH$ is defined as $\mH=\{f\in \mathcal{L}^2(\SS^d):~\|f\|_{\mH}<\infty\}$, with norm squared
    \begin{equation}\label{eqn:Sob_norm_Parseval}
        \|f\|_{\mH}^2=\|f\|_{\mathcal{L}^2(\SS^d)}^2+\sul_{m=1}^\infty m^{2\rr}\|\Pi_m f\|_{\mathcal{L}^2(\SS^d)}^2 = \sul_{m=0}^\infty\sul_{\ell=1}^{N(m)}(m^{2\rr}+1)|\widehat f(m,\ell)|^2.
    \end{equation}
\end{definition}


\subsection{Legendre polynomial and Legendre expansion of $\sigma_k$}
Define the space $\mathcal{L}^2_{w_d}([-1,1])$ by
\begin{equation}
    \lt<f,g\rt>_{w_d}=\int_{-1}^1f(t)g(t)(1-t^2)^{\frac{d-2}{2}}dt,\qquad\|f\|_{\mathcal{L}^2_{w_d}([-1,1])}=\lt<f,f\rt>_{w_d}^{\frac{1}{2}}.
\end{equation}
The space $\mathcal{L}^2_{w_d}([-1,1])$ has an orthogonal polynomial basis $\{p_m\}_{m=0}^\infty$ with $\mathrm{deg}(p_m)=m$ satisfying
\begin{equation}\label{eqn:def_Legendre}
    \lt<p_m,p_l\rt>_{w_d}=0,\quad l\neq m,\quad m\in\NN.
\end{equation}
Such polynomials are called Legendre polynomials, which are known to necessarily have the form (see, e.g., \cite{szego1975orthogonal})
\begin{equation}\label{eqn:lambdam_pm}
    p_m(t)=\lambda_m(1-t^2)^{-\frac{d-2}{2}}\lrt{\frac{d}{dt}}^m\lt[(1-t^2)^{m+\frac{d-2}{2}}\rt],\qquad t\in[-1,1].
\end{equation}
In this paper, we choose the normalization factors $\{\lambda_m\}_{m=0}^\infty$ properly such that \eqref{eqn:sum_Y_nl} below holds.

The integration formula (see, e.g., \cite[Lemma A.5.2.]{dai2013approximation}) shows a univariate function $f\in \mathcal{L}^2_{w_d}([-1,1])$ has the property 
\begin{equation}\label{eqn:int_sph_to_interval}
    \omega_{d-1}\int_{-1}^1f(t)(1-t^2)^{\frac{d-2}{2}}dt = \omega_d\fint_{\SS^d}f(\theta\cdot \eta)d\eta, \qquad \theta\in\SS^d.
\end{equation}

By \cite[Theorem 1.2.6]{dai2013approximation}, there exist Legendre polynomials $\{p_m\}_{m=0}^\infty$ with $\mathrm{deg}(p_m)=m$ such that
\begin{equation}\label{eqn:sum_Y_nl}
p_m(\eta\cdot \theta) = \sum_{\ell=1}^{N(m)} Y_{m,\ell}(\eta)Y_{m,\ell}(\theta).
\end{equation}

With \eqref{eqn:int_sph_to_interval} and \eqref{eqn:sum_Y_nl}, the polynomials $\{p_m\}_{m=0}^\infty$ form an orthogonal basis with respect to the weights $(1-t^2)^{\frac{d-2}{2}}$. Together with \eqref{eqn:def_Legendre}, we call $\{p_m\}_{m=0}^\infty$ in \eqref{eqn:sum_Y_nl} to be the Legendre polynomials throughout this paper.

We remark that \eqref{eqn:sum_Y_nl} are not standard Legendre polynomials as they are not normalized to have norm equal to $1$. The norms of $\{p_m\}_{m=0}^\infty$ can be determined by using the integration formula \eqref{eqn:int_sph_to_interval}:
\begin{equation}\label{eqn:Pn_normalization}
    \begin{split}
         \lt\|p_m\rt\|_{\mathcal{L}^2_{w_d}([-1,1])}=&\int_{-1}^1p_m(t)^2(1-t^2)^{\frac{d-2}{2}}dt= \frac{\omega_{d}}{\omega_{d-1}}\fint_{\SS^d} p_m(e_1\cdot \eta)^2d\eta\\
         =&\frac{\omega_{d}}{\omega_{d-1}}\fint_{\SS^d} \lrt{\sum_{\ell=1}^{N(m)} Y_{m,\ell}(e_1)Y_{m,\ell}(\eta)}^2d\eta
         =\frac{\omega_{d}}{\omega_{d-1}}\sum_{\ell=1}^{N(m)} Y_{m,\ell}(e_1)^2=\frac{\omega_{d}}{\omega_{d-1}}p_m(1)\\
         =&\frac{\omega_{d}}{\omega_{d-1}}\fint_{\SS^d}p_m(1)d\eta=\frac{\omega_{d}}{\omega_{d-1}}\fint_{\SS^d}\sul_{\ell=1}^{N(m)}Y_{m,\ell}(\eta)^2d\eta=\frac{\omega_{d}}{\omega_{d-1}}N(m).
    \end{split}
\end{equation}
Although we will not need them, the normalization factors $\{\lambda_m\}_{m=0}^\infty$ in \eqref{eqn:lambdam_pm} can be computed by comparing \eqref{eqn:Pn_normalization} with the norm of the standard Legendre polynomials (see, e.g., \cite[Chapter 4.3]{szego1975orthogonal}),
\begin{equation*}
    \lambda_m=\frac{\omega_{d}}{\omega_{d-1}}\frac{N(m)}{\Gamma(m+d/2)}\sqrt{\frac{(2m+d-1)\Gamma(m+d-1)}{2^{2m+d-1}\Gamma(m+1)}}.
\end{equation*}

The function $\sigma_k\in \mathcal{L}^2_{w_d}([-1,1])$ has the Legendre expansion in terms of the orthogonal basis $\{p_m\}_{m=0}^\infty$ as
\begin{equation}\label{eqn:Legendre_expansion_ReLUk}
\sigma_k=\sul_{m=0}^\infty\widehat{\sigma_k}(m)p_m,
\end{equation}
where the Legendre coefficients are given as
$$\widehat{\sigma_k}(m)=\frac{\lt<p_m,\sigma_k\rt>_{w_d}}{\|p_m\|_{\mathcal{L}^2_{w_d}([-1,1])}^2}.$$
The coefficients $\{\widehat{\sigma_k}(n)\}_{n=0}^\infty$ are studied in \cite{schneider1967problem,bourgain2006projection,mhaskar2006weighted,bach2017breaking}.
Denote the set
\begin{equation}
    E_{\sigma_k}:=\lt\{m\in\NN:~\widehat{\sigma_k}(m)\neq0\rt\},
\end{equation}
then by \cite[Appendix D.2]{bach2017breaking},
\begin{equation}\label{eqn:hat_sigma_large}
    \begin{split}
        &E_{\sigma_k}=\lt\{m\geq k+1:~m-k\hbox{ is odd}\rt\}\cup\{0,\dots,k\},\\
        &\widehat{\sigma_k}(m)=\frac{\omega_{d-1}k!\Gamma(d/2)}{\omega_d}\frac{(-1)^{(m-k-1)/2}\Gamma(m-k)}{2^m\Gamma\lrt{\frac{m-k+1}{2}}\Gamma\lrt{\frac{m+d+k+1}{2}}},\qquad m\in E_{\sigma_k}.
    \end{split}
\end{equation}

To proceed, we introduce the standard notation of forward and backward difference. Given any $K\in\NN$ and a sequence $\{\mathfrak{a}(m)\}_{m=K}^\infty$, we denote the forward difference $\{(\Delta\mathfrak{a})(m)\}_{m=K}^\infty$ by
        \begin{equation}
            (\Delta\mathfrak{a})(m)=\mathfrak{a}(m+1)-\mathfrak{a}(m),\qquad m\geq K
        \end{equation}
        and the $\beta$-th forward difference $\{(\Delta^\beta\mathfrak{a})(m)\}_{m=K}^\infty$ by
        \begin{equation}
            (\Delta^\beta\mathfrak{a})(m)=\underbrace{\Delta\circ\dots\circ\Delta}_{\beta}\circ \mathfrak{a}(m)=\sul_{j=0}^\beta\binom{\beta}{j}(-1)^{\beta-j}\mathfrak{a}(m+j),\qquad m\in\NN.
        \end{equation}
        Similarly, writing $\mathfrak{a}(K-1):=0$, we denote the backward difference $\{(\nabla\mathfrak{a})(m)\}_{m=K}^\infty$ by
        \begin{equation}
            (\nabla\mathfrak{a})(m)=\mathfrak{a}(m)-\mathfrak{a}(m-1),\qquad m\geq K
        \end{equation}
        and $\{(\nabla^\beta\mathfrak{a})(m)\}_{m=K}^\infty$ by
        \begin{equation}
            (\nabla^\beta\mathfrak{a})(m)=\underbrace{\nabla\circ\dots\circ\nabla}_{\beta}\circ \mathfrak{a}(m),\qquad m\geq K.
        \end{equation}
        It is easy to see the inverses of $\nabla$ and $\nabla^\beta$ are
        \begin{equation}
            \begin{split}
                &\lrt{\nabla^{-1}\mathfrak{a}}(m)=\sul_{\nu=0}^m\mathfrak{a}(\nu),\\
                &\lrt{\nabla^{-\beta}\mathfrak{a}}(m)=\underbrace{\nabla^{-1}\circ\dots\circ\nabla^{-1}}_{\beta}\circ \mathfrak{a}(m)=\sul_{\nu=0}^m\binom{m+\beta-\nu}{\beta}\mathfrak{a}(\nu)
            \end{split}
        \end{equation}

We have the following lemma for the coefficients of $\sigma_k$. 
\begin{lemma}\label{lem:Bach}
Let $\rr\leq\frac{d+2k+1}{2}$, there exists a function $\xi:[k+1,\infty)\to[0,\infty)$ such that
    \begin{equation}\label{eqn:def_theta}
        \begin{split}
            &\xi(m)=\widehat\sigma_k(m)^2m^{2\rr},\qquad m\geq k+1,~m\in E_{\sigma_k},\\
            &0\leq(-1)^\beta(\Delta^\beta\xi)(m)\lesssim m^{2\rr-(d+2k+1)-\beta},\qquad \beta=0,1,\dots
        \end{split}
    \end{equation}
    In particular, if $\rr<\frac{d+2k+1}{2}$,
    \begin{equation}\label{eqn:theta_beta}
        (-1)^\beta(\Delta^\beta\xi)(m)\simeq m^{2\rr-(d+2k+1)-\beta},\qquad \beta=0,1,\dots.
    \end{equation}
\end{lemma}

\begin{proof}

Applying the Legendre duplication formula, 
\begin{equation*} 
        \frac{\Gamma(m-k)}{2^m\Gamma\lrt{\frac{m-k}{2}}\Gamma\lrt{\frac{m-k+1}{2}}}\\
        =\frac{1}{2^{k+1}\sqrt{\pi}},\qquad m-k\notin-\frac{\NN}{2},
\end{equation*}
we have
\begin{equation*}
    \begin{split}
        \frac{\Gamma(m-k)}{2^m\Gamma\lrt{\frac{m-k+1}{2}}\Gamma\lrt{\frac{m+d+k+1}{2}}}=\frac{1}{2^{k+1}\sqrt{\pi}}\frac{\Gamma\lrt{\frac{m-k}{2}}}{\Gamma\lrt{\frac{m+d+k+1}{2}}},\qquad m\geq k+1.
    \end{split}
\end{equation*}
Denote the function
\begin{equation}\label{eqn:theta}
    \xi(t)=\lrt{\frac{\omega_{d-1}}{\omega_d}\frac{k!\Gamma(d/2)}{2^{k+1}\sqrt{\pi}}}^2 t^{2\rr} \lrt{\frac{\Gamma\lrt{\frac{t-k}{2}}}{\Gamma\lrt{\frac{t+d+k+1}{2}}}}^2,
\end{equation}
then with \eqref{eqn:hat_sigma_large}, we have
\begin{equation*}
    \xi(m)=\widehat\sigma_k(m)^2m^{2\rr},\qquad m\geq k+1,~m\in E_{\sigma_k}.
\end{equation*}
Now for $t>k$,
\begin{equation}
    \begin{split}
        \frac{\xi'(t)}{\xi(t)}=&\frac{d}{dt}\log\xi(t)=\frac{2\rr}{t} +\psi\lrt{\frac{t-k}{2}} -\psi\lrt{\frac{t+d+k+1}{2}}
    \end{split}
\end{equation}
where $\psi$ is the digamma function $\displaystyle\psi(t)=\frac{d}{dt}\lrt{\log\Gamma(t)}$. The derivatives of $\psi$ are called polygamma functions and have  series representations
\begin{equation*}
    \psi^{(j)}(t)=(-1)^{j+1}j!\sul_{\nu=0}^\infty(t+\nu)^{-(j+1)}\approx(-1)^{j+1}(j-1)!t^{-j},
\end{equation*}
where the notation $\approx$ here signifies $\displaystyle\lim\limits_{t\to\infty}\frac{\psi^{(j)}(t)}{(-1)^{j+1}(j-1)!t^{-j}}=1$.

We prove \eqref{eqn:def_theta} by induction. Suppose for each $j=0,\dots,r$ we have
\begin{equation}\label{eqn:lim_theta_j}
    \begin{split}
        (d+2k+1-2\rr)^j\leq&\lim\limits_{t\to\infty}\frac{(-1)^j\xi^{(j)}(t)}{t^{2\rr-(d+2k+1)-j}}\leq (d+2k+j+1-2\rr)^j,\\
        &(-1)^j\xi^{(j)}(t)\geq0,\qquad t\geq k+1,
    \end{split}
\end{equation}
then
{\small\begin{equation*}
    \begin{split}
        &\xi^{(\beta+1)}(t)=\lrt{\frac{d}{dt}}^{\beta}\lt[\xi(t)\frac{\xi'(t)}{\xi(t)}\rt]=\sul_{j=0}^\beta\binom{\beta}{j}\xi^{(\beta-j)}(t)\lrt{\frac{d}{dt}}^{j}\lrt{\frac{\xi'(t)}{\xi(t)}}\\
        =&\sul_{j=0}^\beta\binom{\beta}{j}\xi^{(\beta-j)}(t)\lrt{2\rr(-1)^jj!t^{-(j+1)}+2^{-j}\psi^{(j)}\lrt{\frac{t-k}{2}}-2^{-j}\psi^{(j)}\lrt{\frac{t+d+k+1}{2}}}\\
        \approx&\sul_{j=0}^\beta\binom{\beta}{j}\xi^{(\beta-j)}(t)\lt[2\rr(-1)^jj!t^{-(j+1)}+2^{-j}(-1)^{j+1}(j-1)!\lrt{\lrt{\frac{t-k}{2}}^{-j}-\lrt{\frac{t+d+k+1}{2}}^{-j}}\rt]\\
        \approx&\sul_{j=0}^\beta\binom{\beta}{j}\xi^{(\beta-j)}(t)\lt[2\rr(-1)^jj!t^{-(j+1)}+(-1)^{j+1}j!(d+2k+1)t^{-(j+1)}\rt]\\
        =&\sul_{j=0}^\beta\binom{\beta}{j}(d+2k+1-2\rr)\frac{(-1)^{\beta-j}\xi^{(\beta-j)}(t)}{t^{2\rr-(d+2k+1)-\beta+j}}(-1)^{\beta+1}t^{2\rr-(d+2k+1)-\beta-1}.
    \end{split}
\end{equation*}}
Then \eqref{eqn:lim_theta_j} yields
\begin{equation*}
    \begin{split}
        &(d+2k+1-2\rr)\sul_{j=0}^\beta\binom{\beta}{j}(d+2k+1-2\rr)^{\beta-j}\leq\lim\limits_{t\to\infty}\frac{(-1)^{\beta+1}\xi^{(\beta+1)}(t)}{t^{2\rr-(d+2k+1)-\beta-1}}\\
        \leq&(d+2k+1-2\rr)\sul_{j=0}^\beta\binom{\beta}{j}(d+2k+\beta-j+1-2\rr)^{\beta-j},
    \end{split}
\end{equation*}
which proves
\begin{equation}
    (d+2k+1-2\rr)^{\beta+1}\leq\lim\limits_{t\to\infty}\frac{(-1)^{\beta+1}\xi^{(\beta+1)}(t)}{t^{2\rr-(d+2k+1)-\beta-1}}\leq (d+2k+\beta+2-2\rr)^{\beta+1}.
\end{equation}
On the other hand,
\begin{equation*}
    \begin{split}
        &\lrt{\frac{d}{dt}}^{j}\lrt{\frac{\xi'(t)}{\xi(t)}}=2\rr(-1)^jj!t^{-(j+1)}+2^{-j}\psi^{(j)}\lrt{\frac{t-k}{2}}-2^{-j}\psi^{(j)}\lrt{\frac{t+d+k+1}{2}}\\
        =&2\rr(-1)^jj!t^{-(j+1)}+\frac{(-1)^{j+1}}{2^j}j!\sul_{\nu=0}^\infty\lrt{\lrt{\frac{t-k}{2}+\nu}^{-(j+1)}-\lrt{\frac{t+d+k+1}{2}+\nu}^{-(j+1)}}
    \end{split}
\end{equation*}
Some calculus estimation yields
\begin{equation*}
    (-1)^j\lrt{\frac{d}{dt}}^{j}\lrt{\frac{\xi'(t)}{\xi(t)}}>0,
\end{equation*}
together with the induction hypothesis \eqref{eqn:lim_theta_j}, it implies
\begin{equation*}
    (-1)^{\beta+1}\xi^{(\beta+1)}(t)=\sul_{j=0}^\beta\binom{\beta}{j}(-1)^{\beta-j}\xi^{(\beta-j)}(t)(-1)^{j+1}\lrt{\frac{d}{dt}}^{j}\lrt{\frac{\xi'(t)}{\xi(t)}}\geq0.
\end{equation*}
This completes the induction and gives
\begin{equation*}
    0\leq(-1)^\beta\xi^{(\beta)}(t)\lesssim t^{2\rr-(d+2k+1)-\beta},\qquad \beta=0,1,\dots
\end{equation*}
and for $\rr<\frac{d+2k+1}{2}$,
\begin{equation*}
    (-1)^\beta\xi^{(\beta)}(t)\simeq t^{2\rr-(d+2k+1)-\beta},\qquad \beta=0,1,\dots.
\end{equation*}
Consequently the formula
\begin{equation*}
    (-1)^\beta(\Delta^\beta\xi)(m)=\int_{m}^{m+1}\underbrace{\int_0^{1}\dots\int_0^{1}}_{\beta-1}(-1)^\beta\xi^{(\beta)}(t_1+\dots+t_\beta)dt_\beta\dots dt_1.
\end{equation*}
proves \eqref{eqn:def_theta} and \eqref{eqn:theta_beta}.

\end{proof}


\subsection{Properties of the polynomials with scattered points}\label{subsec:prop_Legendre}
In this subsection, we consider a finite subset $\{\theta_j^*\}_{j=1}^n\subset \mathbb{S}^{d}$ comprising distinct, scattered points. The mesh size for $\{\theta_j^*\}_{j=1}^n\subset \mathbb{S}^{d}$ is defined by

\begin{equation}
h=\mal_{\eta\in\SS^d}\min\limits_{1\leq j\leq n}\rho(\eta,\theta_j^*).
\end{equation}

The existence of positive quadrature rules on $\SS^d$ based on scattered points is a known result in approximation theory (see, e.g., \cite{mhaskar2001spherical,bondarenko2013optimal,jetter2023norming}). For our case, one could construct a $\{\theta_j^*\}_{j=1}^n$-compatible decomposition as described in \cite{mhaskar2001spherical} based on the spherical shells $\BB_\rho(\theta_j^*,h)\setminus\BB_\rho(\theta_j^*,\underline{h}/2)$. Then the following lemma follows straightforwardly by \cite[Theorem 4.1]{mhaskar2001spherical}.
\begin{lemma}[Theorem 4.1. \cite{mhaskar2001spherical}]\label{lem:quadrature}
    Given scattered points $\{\theta_j^*\}_{j=1}^n \subset \SS^d$ with mesh norm $h$, there exist nonnegative weights $\tau_1,\dots,\tau_n$ with $\tau_j\lesssim h^d$ and a constant $C_1$ (independent of $n, h$) such that
    \begin{equation}\label{eqn:quadrature}
        \fint_{\SS^d}p(\eta)q(\eta)d\eta=\sul_{j=1}^n\tau_jp(\theta_j^*)q(\theta_j^*),\quad \forall p,q\in\mathbb{P}_J(\SS^d),
    \end{equation}
    where $J=\lfloor C_1h^{-1}\rfloor$.
\end{lemma}

In this paper, the matrices induced by Legendre polynomials, defined as
\begin{equation}\label{eqn:def_mathcalPn0}
P(m) = \lrt{p_m(\theta_i^* \cdot \theta_j^*)}_{i,j=1}^n,
\end{equation}
plays a crucial role in our analysis. In this subsection, we summarize the key properties of these matrices, which are essential for the proof of our main result and will be utilized in the next section.

   For $\beta\in\NN$, we denote sequences of matrices $\lt\{P_\beta(m)\rt\}_{m=0}^\infty$ inductively by
    \begin{equation}\label{eqn:p_n^r_explicit}
        P_{\beta+1}(m)=\sul_{\nu=0}^mP_{\beta}(\nu)=\lrt{\nabla^{-\beta}P}(m)=\sul_{\nu=0}^m\binom{m+\beta-\nu}{\beta}P(\nu),\qquad m\in\NN.
    \end{equation}
    We also allow the summation to begin at indices other than $0$. Given the sequence $\{P_\beta(m)\}_{m=K}^\infty$, we denote the sequences $\{P_{K,\beta+1}(m)\}_{m=K}^\infty$ by
    \begin{equation}
        P_{K,\beta+1}(m)=(\nabla^{-1}P_\beta)(m)=\sul_{\nu=K}^mP_\beta(\nu),\qquad m\geq K.
    \end{equation}

An important property of the Legendre polynomials is that the Ces\`aro summations, $\{P_\beta(m)\}_{m=0}^\infty$, become increasingly ``diagonally dominated" as $\beta$ grows (see, e.g., \cite{chanillo1993weak}). Specifically, the ratio
$\frac{\lt|\lrt{P_\beta(m)}_{i,i}\rt|}{\sum_{i\neq j}\lt|\lrt{P_\beta(m)}_{i,j}\rt|}$
increases with $\beta$. This behavior indicates that the Ces\`aro summations of $\lt\{P_\beta(m)\rt\}_{m=0}^\infty$ exhibit a highly localized property, which plays a crucial role in studying approximation properties (see, e.g., \cite[Chapter 2.6, Chapter 11.4]{dai2013approximation}). A similar phenomenon arises in Fourier analysis, where the Ces\`aro summations of the Dirichlet kernels, known as Fejér kernels, exhibit localization properties (see, e.g., \cite[Chapter 7]{devore1993constructive}). This highly-localized property is also significant in analyzing kernel approximation properties (see \cite{mhaskar2010eignets, mhaskar2020kernel}).

Lemma \ref{lem:property_p_nr} establishes several key properties of $P_\beta(m)$, enabling sharp estimations in the proof presented in Section \ref{sec:pf_sphere}. Before proving Lemma \ref{lem:property_p_nr}, we restate a related result from \cite[Corollary 14.7]{chanillo1993weak} using our notation.

\begin{lemma}\label{lem:est_p_nij}
For $\beta\geq2$ and all $i\neq j$,
    \begin{equation}
        \lt|(P_\beta(m))_{i,j}\rt|\lesssim\frac{m^{\frac{d-1}{2}}}{\rho(\theta_i^*,\theta_j^*)^{\frac{d-1}{2}+\beta}\lrt{\max\lt\{\rho(\theta_i^*,-\theta_j^*),\frac{1}{m}\rt\}}^{\frac{d-1}{2}}}.
    \end{equation}
    where the corresponding constant is only dependent of $d$.
\end{lemma}

\begin{proof}
    Let $\Bigl\{p_n^{(\frac{d-2}{2},\frac{d-2}{2})}\Bigr\}_{n=0}^\infty$ be the Legendre polynomials defined in \cite{chanillo1993weak}, then each $p_n^{(\frac{d-2}{2},\frac{d-2}{2})}$ must equal to $p_n$ multiplied by a real number. The function $L_n^{(\frac{d-2}{2},\frac{d-2}{2}),\beta}$ is denoted in \cite[(2.27)]{chanillo1993weak} by, for $t\in[-1,1]$,
    \begin{equation}
        \begin{split}
            \binom{n+\beta-1}{\beta-1}L_n^{(\frac{d-2}{2},\frac{d-2}{2}),\beta}(t,1)=&\sul_{\nu=0}^n\binom{n+\beta-1-\nu}{\beta-1}\frac{ p_\nu^{(\frac{d-2}{2},\frac{d-2}{2})}(t) p_\nu^{(\frac{d-2}{2},\frac{d-2}{2})}(1)}{\| p_\nu^{(\frac{d-2}{2},\frac{d-2}{2})}\|_{L_{w_d}^2([-1,1])}^2}\\
            =&\sul_{\nu=0}^n\binom{n+\beta-1-\nu}{\beta-1}\frac{ p_\nu(t) p_\nu(1)}{\| p_\nu\|_{L_{w_d}^2([-1,1])}^2}\\
            =&\frac{\omega_{d-1}}{\omega_d}\sul_{\nu=0}^n\binom{n+\beta-1-\nu}{\beta-1} p_\nu(t)=\frac{\omega_{d-1}}{\omega_d}p_{n,\beta}(t).
        \end{split}
    \end{equation}
    where the last equality follows by \eqref{eqn:Pn_normalization}.
By \cite[Corollary 14.7]{chanillo1993weak},
    \begin{equation}
    \begin{split}
        &\lt|L_n^{(\frac{d-2}{2},\frac{d-2}{2})}(t,1)\rt|\\
        \lesssim&\lrt{n\lrt{\sqrt{\frac{1-t}{2}}+\frac{1}{n}}^{d-1}\left(\sqrt{1-t}+\frac{1}{n}\right)^2}^{-1}\\
        &+\lrt{n^{\beta-1}\lrt{\sqrt{1-t}+\frac{1}{n}}^{\frac{d-1}{2}}\lrt{\sqrt{1+t}+\frac{1}{n}}^{\frac{d-1}{2}}\lrt{\frac{1}{n}}^{\frac{d-1}{2}}\left(\sqrt{1-t}+\frac{1}{n}\right)^\beta}^{-1}
    \end{split}
\end{equation}
By noticing the formula
$$\sqrt{1-\theta_i\cdot \theta_j}=\sqrt{1-\cos(\rho(\theta_i,\theta_j))}=\sqrt{2}\sin\lrt{\frac{\rho(\theta_i,\theta_j)}{2}},\qquad \theta_i\cdot \theta_j\geq0,$$
we have
$$\sqrt{1-\theta_i\cdot \theta_j}\simeq\rho(\theta_i,\theta_j).$$
As $n\geq M\gtrsim\underline{h}$, we have
$$\rho(\theta_i,\theta_j)\geq\underline{h}\gtrsim\frac{1}{n},$$
then
\begin{equation*}\label{eqn:P_ij_s>1}
    \begin{split}
        \lt| p_{n,\beta}(\theta_i\cdot \theta_j)\rt|\lesssim&n^{\beta-1}\lrt{\lrt{n\rho(\theta_i,\theta_j)^{d+1}}^{-1}+\lrt{n^{\beta-1}\rho(\theta_i,\theta_j)^{\frac{d-1}{2}}\rho(\theta_i,-\theta_j)^{\frac{d-1}{2}}\lrt{\frac{1}{n}}^{\frac{d-1}{2}}\rho(\theta_i,\theta_j)^\beta}^{-1}}\\
        \lesssim&\frac{n^{\frac{d-1}{2}}}{\rho(\theta_i,\theta_j)^{\frac{d-1}{2}+\beta}\lrt{\max\lt\{\rho(\theta_i,-\theta_j),\frac{1}{n}\rt\}}^{\frac{d-1}{2}}}.
    \end{split}
\end{equation*}
\end{proof}

    \begin{lemma}\label{lem:property_p_nr}
        Let $r\in\NN$, the matrices $\lt\{P_\beta(m)\rt\}$ have the following properties:
        \begin{itemize}
            \item [(a)] Each $P_\beta(m)$ is semi-positive definite.
            \item [(b)] The diagonal elements of $P_\beta(m)$ are equal and satisfy
            \begin{equation}\label{eqn:diag_p_nr}
                \lrt{P_\beta(m)}_{i,i}=\max\limits_{1\leq i,j\leq n}\lt|\lrt{P_\beta(m)}_{i,j}\rt|\simeq m^{\beta+d-1},
            \end{equation}
            \item [(c)] Let $\alpha=\lceil\frac{d+2}{2}\rceil$ and $J$ be the integer in Lemma \ref{lem:quadrature},
        \begin{equation}\label{eqn:est_Pns_2norm}
        \lt\|P_\alpha(m)\rt\|_2\lesssim m^{\alpha+d-1}\lrt{\frac{h}{\underline{h}}}^{\frac{d-1}{2}+\alpha},\quad m\geq J+1,
    \end{equation}
    where $\underline{h}:= \min\limits_{i\neq j} \rho(\theta_i^*,\theta_j^*)$. 
        \end{itemize}
    \end{lemma}
    \begin{proof}
        To prove (a), it suffices to show each $P(m)$ is semi-positive definite. For any vector $a\in\RR^n$, by \eqref{eqn:sum_Y_nl},
        \begin{equation}\label{eqn:Pn0=YY}
            \begin{split}
                a^\top P(m)a=&\sul_{1\leq i,j\leq n}a_ia_jp(\theta_i^*\cdot \theta_j^*)=\sul_{1\leq i,j\leq n}\sul_{\ell=1}^{N(m)}a_ia_jY_{m,\ell}(\theta_i^*)Y_{m,\ell}(\theta_j^*)\\
                =&\sul_{\ell=1}^{N(m)}\lrt{\sul_{j=1}^na_jY_{m,\ell}(\theta_j^*)}^2\geq0.
            \end{split}
        \end{equation}
        This proves $P(m)$ is semi-positive definite and consequently (a).

        To prove (b), we apply \eqref{eqn:sum_Y_nl}, \eqref{eqn:Pn_normalization} and write
\begin{equation}\label{eqn:p_n_max}
\begin{split}
    |p_m(\eta\cdot \theta)| &= \left|\sum_{\ell=1}^{N(m)} Y_{m,\ell}(\eta)Y_{m,\ell}(\theta)\right|\\ &\leq \left(\sum_{\ell=1}^{N(m)} Y_{m,\ell}(\eta)^2\right)^{1/2}\left(\sum_{\ell=1}^{N(m)} Y_{m,\ell}(\theta)^2\right)^{1/2}\\
    & = p_m(1)= N(m).
\end{split}
\end{equation}
Thus $\lrt{P(m)}_{i,j}=p_m(\theta_i^*\cdot \theta_j^*)$ attains its maximum value $p_m(1)=N(m)$ when $i=j$. Therefore,
        \begin{equation}
            \begin{split}
                \max\limits_{1\leq i,j\leq n}\lt|\lrt{P_\beta(m)}_{i,j}\rt|=&\lrt{P_\beta(m)}_{i,i}=\lt|\sul_{\nu=0}^m\binom{m+\beta-1-\nu}{\beta-1}\lrt{P_{\nu,0}}_{i,i}\rt|\\
                =&\sul_{\nu=0}^m\binom{m+\beta-1-\nu}{\beta-1}N(\nu)\simeq m^{\beta+d-1},
            \end{split}
        \end{equation}
        To prove (c), we apply Lemma \ref{lem:est_p_nij}. In our case, $\alpha=\lceil\frac{d+2}{2}\rceil\geq2$, so Lemma \ref{lem:est_p_nij} holds true. Given this lemma, we divide the set $\{\theta_i^*:~1\leq i\leq n,~i\neq j\}$ in terms of the distance to $\theta_j^*$ and $-\theta_j^*$ as
\begin{equation*}
    \{\theta_i:~1\leq i\leq n,~i\neq j\}=\mathcal{I}_{-1,j}\cup\bigcup\limits_{p=0}^{\lceil\log_2\lrt{\frac{\pi}{2\underline{h}}}\rceil}\lrt{\mathcal{I}_{p,j,+}\cup\mathcal{I}_{p,j,-}},
\end{equation*}
where $\mathcal{I}_{-1,j}:=\lt\{i:\rho(\theta_i^*,-\theta_j^*)<\underline{h}\rt\}$ and for $p=0,1,\dots$,
$$\mathcal{I}_{p,j,+}:=\{i:2^p\underline{h}\leq\rho(\theta_i^*,\theta_j^*)<2^{p+1}\underline{h}\},\quad\mathcal{I}_{p,j,-}:=\{i:2^p\underline{h}\leq\rho(\theta_i^*,-\theta_j^*)<2^{p+1}\underline{h}\}.$$
By a measure argument, it is easy to verify
$$\#\mathcal{I}_{-1,j}\lesssim1,\quad\#\mathcal{I}_{p,j,+}\lesssim2^{pd},\quad\#\mathcal{I}_{p,j,-}\lesssim2^{pd}$$
where the corresponding constants are only dependent of $d$.

By Lemma \ref{lem:est_p_nij},
\begin{equation*}
    \lt|\lrt{P_{\alpha}(m)}_{i,j}\rt|\lesssim\frac{m^{\frac{d-1}{2}}}{\rho(\theta_i^*,\theta_j^*)^{\frac{d-1}{2}+\alpha}\lrt{\max\lt\{\rho(\theta_i^*,-\theta_j^*),\frac{1}{m}\rt\}}^{\frac{d-1}{2}}},
\end{equation*}
then we can write
\begin{equation}\label{eqn:P_ij}
    \begin{split}
        \sul_{j\neq i}\lt|\lrt{P_{\alpha}(m)}_{i,j}\rt|\lesssim& \sul_{p=0}^{\lceil\log_2\lrt{\frac{\pi}{2\underline{h}}}\rceil}\sul_{i\in\mathcal{I}_{p,j,+}\cup\mathcal{I}_{p,j,-}}\frac{m^{\frac{d-1}{2}}}{\rho(\theta_i^*,\theta_j^*)^{\frac{d-1}{2}+\alpha}\rho(\theta_i^*,-\theta_j^*)^{\frac{d-1}{2}}}+\sul_{i\in\mathcal{I}_{-1,j,-}}\frac{m^{d-1}}{\rho(\theta_i^*,\theta_j^*)^{\frac{d-1}{2}+\alpha}}\\
        \lesssim&\sul_{p=0}^{\lceil\log_2\lrt{\frac{\pi}{2\underline{h}}}\rceil}\sul_{i\in\mathcal{I}_{p,j,+}\cup\mathcal{I}_{p,j,-}}(2^p\underline{h})^{-\lrt{\frac{d-1}{2}+\alpha}}m^{\frac{d-1}{2}}+\sul_{i\in\mathcal{I}_{-1,j}}m^{d-1}\\
        \lesssim&\sul_{p=0}^{\lceil\log_2\lrt{\frac{\pi}{2\underline{h}}}\rceil}2^{pd}(2^p\underline{h})^{-\lrt{\frac{d-1}{2}+\alpha}}m^{\frac{d-1}{2}}+m^{d-1}\\
        \lesssim&m^{\frac{d-1}{2}}\underline{h}^{-(\frac{d-1}{2}+\alpha)}+m^{d-1},
    \end{split}
\end{equation}
where the last inequality follows by recalling $\alpha=\lceil\frac{d+2}{2}\rceil$.

Together with \eqref{eqn:diag_p_nr}, we get
    \begin{equation}\label{eqn:P_row_sum}
        \begin{split}
            \sul_{j=1}^n\lt|\lrt{P_{\alpha}(m)}_{i,j}\rt|\lesssim m^{d-1}+m^{\frac{d-1}{2}}\underline{h}^{-(\frac{d-1}{2}+\alpha)}+m^{\alpha+d-1}\lesssim m^{\alpha+d-1}\lrt{\frac{h}{\underline{h}}}^{\frac{d-1}{2}+\alpha},\quad m\geq J+1.
        \end{split}
    \end{equation}
    Thus we can estimate the matrix norm of $P_{\alpha}(m)$ as
    \begin{equation}
        \lt\|P_{\alpha}(m)\rt\|_2\leq\sqrt{\lt\|P_{\alpha}(m)\rt\|_1\lt\|P_{\alpha}(m)\rt\|_\infty}\lesssim m^{\alpha+d-1}\lrt{\frac{h}{\underline{h}}}^{\frac{d-1}{2}+\alpha},\quad m\geq J+1.
    \end{equation}
    \end{proof}

\section{Approximation and characterization theorems for Sobolev spaces on spheres}\label{sec:pf_sphere}

In this section, we establish a spherical counterpart of Theorem \ref{thm:appr_rate_ball} and \ref{thm:barron_sob_equ}. We begin by addressing the spherical case because the function $\sigma_k(\theta\cdot\eta) \in \mathcal{L}^2(\mathbb{S}^d)$ admits a spherical harmonic expansion, whereas obtaining such an expansion for functions of the form $\sigma_k(w\cdot x + b)$ is challenging (as they are not periodic and prevents an application of Fourier series). This approach is partially inspired by \cite{bach2017breaking} and \cite{yang2024optimal}.

\subsection{Approximating functions on spheres by ReLU$^k$ linear vector spaces}
\begin{theorem}\label{thm:appr_rate_sph}
    Let $d,n\in\NN$, $k\in\NN_0$, $\rr\in(0,\frac{d+2k+1}{2}]$, $\ss\leq\min\{k,\rr\}$, and $\{\theta_j^*\}_{j=1}^n\subset\SS^d$. Then for any $f\in\mathcal{H}^\rr(\SS^d)$ satisfying
    \begin{equation}\label{eqn:f_evenodd_relu}
        f(\eta)=(-1)^{k+1}f(-\eta),\qquad \eta\in\SS^d,
    \end{equation}
    there exists $a\in\RR^n$ with $\|a\|_2\lesssim h^{-\frac{2k+1-2\rr}{2}}\|f\|_{\mathcal{H}^{\rr}(\SS^d)}$ such that
    \begin{equation}\label{eqn:rate_nonuniform}
        \Bigl\|f-\sul_{j=1}^n a_j\sigma_k(\theta_j^*\cdot\circ)\Bigr\|_{\mathcal{H}^\ss(\SS^d)}\lesssim h^{\rr-\ss}\lt\|f\rt\|_{\mathcal{H}^\rr(\SS^d)}.
    \end{equation}
    where
    $$h=\mal_{\eta\in\SS^d}\min\limits_{1\leq j\leq n}\rho(\eta,\theta_j^*).$$
    If the collection $\{\theta_j^*\}_{j=1}^n\subset\SS^d$ is well-distributed, for any $f\in\mathcal{H}^\rr(\SS^d)$, there exists $a\in\RR^n$ with $\|a\|_2\lesssim n^{\frac{2k+1-2\rr}{2d}}\|f\|_{\mathcal{H}^{\rr}(\SS^d)}$ such that
    \begin{equation}\label{eqn:rate_uniform}
        \Bigl\|f-\sul_{j=1}^n a_j\sigma_k(\theta_j^*\cdot\circ)\Bigr\|_{\mathcal{H}^\ss(\SS^d)}\lesssim h^{\rr-\ss}\lt\|f\rt\|_{\mathcal{H}^\rr(\SS^d)}\simeq n^{-\frac{\rr-\ss}{d}}\lt\|f\rt\|_{\mathcal{H}^\rr(\SS^d)}.
    \end{equation}
    With $\rr=\frac{d+2k+1}{2}$ and $\ss=0$, there exists $a\in\RR^n$ with $\|a\|_2\lesssim n^{-\frac{1}{2}}\|f\|_{\mathcal{H}^{\frac{d+2k+1}{2}}(\SS^d)}$ such that
    \begin{equation}
        \Bigl\|f-\sul_{j=1}^n a_j\sigma_k(\theta_j^*\cdot\circ)\Bigr\|_{\mathcal{L}^2(\SS^d)}\lesssim n^{-\frac{1}{2}-\frac{2k+1}{2d}}\lt\|f\rt\|_{\mathcal{H}^{\frac{d+2k+1}{2}}(\SS^d)}.
    \end{equation}
    All the corresponding constants are independent of $n,\{\theta_j^*\}_{j=1}^n$, and $f$.
\end{theorem}

\begin{proof}

To start up, there exists a subset of $\{\vartheta_j^*\}_{j=1}^{\tilde n}\subset\{\theta_j^*\}_{j=1}^n$ with $\tilde n\simeq n$ such that
    \begin{equation}\label{eqn:quasi-uniform}
        \mal_{\eta\in\SS^d}\min\limits_{1\leq j\leq \tilde n}\rho(\eta,\vartheta_j^*)\lesssim\min\limits_{i\neq j}\rho(\vartheta_i^*,\vartheta_j^*).
    \end{equation}
    Thus it suffices to prove the theorem for $\{\theta_j^*\}_{j=1}^n$ satisfying $h\lesssim\underline{h}$.

For each $j=1,\dots,n$, we can write the expansion \eqref{eqn:Legendre_expansion_ReLUk} as
\begin{equation}
    \sigma_k(\theta_j^*\cdot \eta)=\sul_{m\in E_{\sigma_k}}\widehat{\sigma_k}(m)p_m(\theta_j^*\cdot \eta)=\sul_{m\in E_{\sigma_k}}\widehat{\sigma_k}(m)\sul_{\ell=1}^{N(m)}Y_{m,\ell}(\theta_j^*)Y_{m,\ell}(\eta)\qquad a.e. \quad \eta\in\SS^d.
\end{equation}
Then with $a=(a_1,\dots,a_n)^\top$ which is to be determined later,
\begin{equation}\label{eqn:f_n_expansion}
    f_n(\eta)=\sul_{j=1}^na_j\sigma_k(\theta_j^*\cdot \eta)=\sul_{m\in E_{\sigma_k}}\sul_{\ell=1}^{N(m)}\widehat{f_n}(m,\ell)Y_{m,\ell}(\eta),\qquad a.e. \quad \eta\in\SS^d.
\end{equation}
where
\begin{equation}\label{eqn:hat_fn_explicit}
    \widehat{f_n}(m,\ell)=\widehat{\sigma_k}(m)\sul_{j=1}^na_jY_{m,\ell}(\theta_j^*).
\end{equation}

On the other hand,
since the Legendre polynomials satisfy
\begin{equation}\label{eqn:legend_even}
    p_m(-t)=(-1)^mp_m(t),\qquad t\in[-1,1],~m\in\NN,
\end{equation}
if $m-k$ is even,
\begin{equation*}
    \begin{split}
        (\Pi_mf)(\theta)=&\fint_{\SS^d}f(-\eta)p_m(-\eta\cdot \theta)d\eta=\fint_{\SS^d}(-1)^{k+1}f(\eta)(-1)^mp_m(\eta\cdot \theta)d\eta\\
        =&-\fint_{\SS^d}f(\eta)p_m(\eta\cdot \theta)d\eta=-(\Pi_mf)(\theta),\qquad \theta\in\SS^d,
    \end{split}
\end{equation*}
which implies $\Pi_mf\equiv0$ and $f\in\bigoplus_{m\in E_{\sigma_k}}\YY_m$.

So we can write the expansion
$$f(\eta)=\sul_{m\in E_{\sigma_k}}\sul_{\ell=1}^{N(m)}\widehat f(m,\ell)Y_{m,\ell}(\eta),\qquad a.e. \quad \eta\in\SS^d.$$

By Lemma \ref{lem:quadrature}, 
there exists nonnegative numbers
\begin{equation}\label{eqn:nu_bound}
    \tau_1,\dots,\tau_n\lesssim h^d
\end{equation}
such that
\begin{equation}\label{eqn:ideal_aj_tofind}
    \begin{split}
    \widehat f(m,\ell)=&\fint_{\SS^d}f(\eta)Y_{m,\ell}(\eta)d\eta=\fint_{\SS^d}(\Pi_{m}f)(\eta)Y_{m,\ell}(\eta)d\eta\\
    =&\fint_{\SS^d}\lrt{\sul_{m'\in E_{\sigma_k,J}}\widehat{\sigma_k}(m')^{-1}(\Pi_{m'}f)(\eta)}\widehat{\sigma_k}(m)Y_{m,\ell}(\eta)d\eta\\
    =&\sul_{j=1}^n\tau_j\lrt{\sul_{m'\in E_{\sigma_k,J}}\widehat{\sigma_k}(m')^{-1}(\Pi_{m'}f)(\theta_j^*)}\widehat{\sigma_k}(m)Y_{m,\ell}(\theta_j^*),\qquad m\leq J.
    \end{split}
\end{equation}
where $J$ is the integer in Lemma \ref{lem:quadrature}, and
\begin{equation*}
    E_{\sigma_k,J}:=\{m\in E_{\sigma_k}:~m\leq J\}.
\end{equation*}
Comparing \eqref{eqn:hat_fn_explicit} and \eqref{eqn:ideal_aj_tofind}, by taking
\begin{equation}
    a_j=\tau_j\lrt{\sul_{m\in E_{\sigma_k,J}}\widehat{\sigma_k}(m)^{-1}(\Pi_mf)(\theta_j^*)},\qquad j=1,\dots,n,
\end{equation}
we have
\begin{equation}\label{eqn:trunc_equal}
    \widehat {f_n}(m,\ell)=\widehat f(m,\ell),\qquad m\leq J.
\end{equation}

Given Lemma \eqref{eqn:trunc_equal}, in terms of \eqref{eqn:Sob_norm_Parseval} we can write
\begin{equation}
    \begin{split}
        \lt\|f-f_n\rt\|_{\mathcal{H}^\ss(\SS^d)}^2
        =&\sul_{m=0}^\infty\sul_{\ell=1}^{N(m)}\lt(\widehat f(m,\ell)-\widehat{f_n}(m,\ell)\rt)^2(m^{2\ss}+1)\\
        =&\sul_{m=J+1}^\infty\sul_{\ell=1}^{N(m)}\lt(\widehat f(m,\ell)-\widehat{f_n}(m,\ell)\rt)^2(m^{2\ss}+1)\\
        \leq&4\sul_{m=J+1}^\infty\sul_{\ell=1}^{N(m)}\widehat f(m,\ell)^2m^{2\ss}+4\sul_{m=J+1}^\infty\sul_{\ell=1}^{N(m)}\widehat{f_n}(m,\ell)^2m^{2\ss}.
    \end{split}
\end{equation}

Write
\begin{equation*}
    \begin{split}
        I_1=\sul_{m=J+1}^\infty\sul_{\ell=1}^{N(m)}\widehat f(m,\ell)^2m^{2\ss},\qquad
        I_2=\sul_{m=J+1}^\infty\sul_{\ell=1}^{N(m)}\widehat {f_n}(m,\ell)^2m^{2\ss},
    \end{split}
\end{equation*}
then
\begin{equation}\label{eqn:est_I1}
    I_1\leq\sul_{m=J+1}^\infty\sul_{\ell=1}^{N(m)}\widehat f(m,\ell)^2m^{2\rr}J^{2\ss-2\rr}\leq J^{2\ss-2\rr}\|f\|_{\mH}^2\simeq h^{2\rr-2\ss}\|f\|_{\mH}^2.
\end{equation}
Let $ P(m)$ be as in \eqref{eqn:Pn0=YY},
\begin{equation}\label{eqn:est_I2}
    \begin{split}
        I_2=&\sul_{m= J+1}^\infty\sul_{\ell=1}^{N(m)}\widehat{\sigma_k}(m)^2\lrt{\sul_{j=1}^na_jY_{m,\ell}(\theta_j^*)}^2m^{2\ss}=\sul_{m=J+1}^\infty\widehat{\sigma_k}(m)^2m^{2\ss}a^\top P(m)a\\
        \leq&\sul_{m=J+1}^\infty\xi(m)a^\top P(m)a,
    \end{split}
\end{equation}
where $\xi$ is the nonnegative function determined in Lemma \ref{lem:Bach}, which satisfies $$\xi(m)=\widehat{\sigma_k}(m)^2m^{2\ss},\qquad m\in E_{\sigma_k},~m\geq J+1.$$

    Summation by parts, we have
    \begin{equation*}
        \begin{split}
            &\sul_{m=J+1}^\infty\xi(m)P(m)=\sul_{m=J+1}^\infty\xi(m)\lrt{\nabla P_{J+1,1}}(m)\\
            =&\lim\limits_{m\to\infty}\xi(m+1)P_{J+1,1}(m)-\sul_{m=J+1}^\infty\lrt{\nabla\xi}(m+1)P_{J+1,1}(m)\\
            =&\lim\limits_{m\to\infty}\xi(m+1)P_{J+1,1}(m)-\sul_{m=J+1}^\infty(\Delta\xi)(m)P_{J+1,1}(m).
        \end{split}
    \end{equation*}
    where we abused the notation that $P_{J+1,1}(J)=\sul_{\nu=J+1}^{J}P_{\nu,0}=0$. Similarly, for $\beta=1,\dots,\alpha-1$,
    \begin{equation*}
        \begin{split}
            \sul_{m=J+1}^\infty(-1)^\beta(\Delta^\beta\xi)(m)P_\beta(m)=&\lim\limits_{m\to\infty}(-1)^\beta(\Delta^\beta\xi)(m+1)P_{J+1,\beta+1}(m)\\
            &+\sul_{m=J+1}^\infty(-1)^{\beta+1}(\Delta^{\beta+1}\xi)(m)P_{J+1,\beta+1}(m).
        \end{split}
    \end{equation*}
    By Lemma \ref{lem:Bach},
    \begin{equation*}
        \begin{split}
            \sul_{m=J+1}^\infty(-1)^\beta(\Delta^\beta\xi)(m)P_\beta(m)=&\sul_{m=J+1}^\infty(-1)^{\beta+1}(\Delta^{\beta+1}\xi)(m)P_{J+1,\beta+1}(m)\\
            \leq&\sul_{m=J+1}^\infty(-1)^{\beta+1}(\Delta^{\beta+1}\xi)(m)P_{\beta+1}(m).
        \end{split}
    \end{equation*}
    Therefore,
    \begin{equation}\label{eqn:I2_est_2}
        I_2\leq\sul_{m=J+1}^\infty\xi(m)a^\top P(m)a\leq\sul_{m=J+1}^\infty(-1)^\alpha(\Delta^\alpha\xi)(m)a^\top P_\alpha(m)a.
    \end{equation}
Together with \eqref{eqn:def_theta} and \eqref{eqn:est_Pns_2norm},
    \begin{equation}\label{eqn:fM-fMm_a_norm}
        \begin{split}
            I_2\lesssim\lrt{\frac{h}{\underline{h}}}^{\frac{d-1}{2}+\alpha}\|a\|_2^2\sul_{m=J+1}^\infty m^{2\ss-2k-2}\simeq\lrt{\frac{h}{\underline{h}}}^{\frac{d-1}{2}+\alpha}\|a\|_2^2h^{2k+1-2\ss},            
        \end{split}
    \end{equation}

Thus
    \begin{equation}\label{eqn:f-Phim_1}
    \begin{split}
        \lt\|f-f_n\rt\|_{\mathcal{H}^\ss(\SS^d)}^2\lesssim h^{2\rr-2\ss}\|f\|_{\mH}^2+\lrt{\frac{h}{\underline{h}}}^{\frac{d-1}{2}+\alpha}\|a\|_2^2h^{2k+1-2\ss}.
    \end{split}
    \end{equation}

    For the norm $\|a\|_2$, we apply \eqref{eqn:nu_bound} and the quadrature formula \eqref{eqn:quadrature} to conclude
    \begin{equation}\label{eqn:a_2_norm}
        \begin{split}
        \|a\|_2^2=&\sul_{j=1}^n\tau_j^2\lrt{\sul_{m\in E_{\sigma_k,J}}\widehat{\sigma_k}(m)^{-1}(\Pi_mf)(\theta_j^*)}^2
        \lesssim h^d\sul_{j=1}^n\tau_j\lrt{\sul_{m\in E_{\sigma_k,J}}\widehat{\sigma_k}(m)^{-1}(\Pi_mf)(\theta_j^*)}^2\\
        =&h^d\fint_{\SS^d}\lrt{\sul_{m\in E_{\sigma_k,J}}\widehat{\sigma_k}(m)^{-1}\sul_{\ell=1}^{N(m)}\widehat f(m,\ell)Y_{m,\ell}(\eta)}^2d\eta=h^d\sul_{m\in E_{\sigma_k,J}}\widehat{\sigma_k}(m)^{-2}\sul_{\ell=1}^{N(m)}\widehat f(m,\ell)^2\\
        \leq&
        h^{d}\mal_{m\in E_{\sigma_k,J}}\lrt{\frac{m^{d+2k+1}}{m^{2\rr}+1}}\sul_{m\in E_{\sigma_k,J}}\sul_{\ell=1}^{N(m)}(m^{2\rr}+1)\widehat f(m,\ell)^2\\
        \lesssim&h^{2\rr-2k-1}\lt\|f\rt\|_{\mathcal{H}^\rr(\SS^d)}^2.
        \end{split}
    \end{equation}
Substituting this into \eqref{eqn:fM-fMm_a_norm},
\begin{equation*}
    \lrt{\frac{h}{\underline{h}}}^{\frac{d-1}{2}+\alpha}\|a\|_2^2h^{2k+1-2\ss}\lesssim\lrt{\frac{h}{\underline{h}}}^{\frac{d-1}{2}+\alpha}h^{2(\rr-\ss)}\lt\|f\rt\|_{\mathcal{H}^\rr(\SS^d)}^2\lesssim h^{2(\rr-\ss)}\lt\|f\rt\|_{\mathcal{H}^\rr(\SS^d)}^2,
\end{equation*}
where the second inequality follows by \eqref{eqn:quasi-uniform}. Together with \eqref{eqn:f-Phim_1}, we get \eqref{eqn:rate_nonuniform} and complete the proof.

\end{proof}

\begin{remark}\label{rem:even_odd_necess}
    We remark here that the condition \eqref{eqn:f_evenodd_relu} is necessary. This is because any $\sigma_k(\theta^*\cdot\eta)$ is essentially an even (odd) function when $k$ is odd (even) up to a polynomial. It suffices to notice
    \begin{equation*}
        \sigma_k(\theta^*\cdot \eta)=(-1)^{k+1}\sigma_k(-\theta^*\cdot \eta)+(\theta^*\cdot \eta)^k.
    \end{equation*}
\end{remark}

\subsection{Integral representation for Sobolev and Barron functions on spheres}
In this subsection, we provide a new characterization of $\mathcal{H}^{\frac{d+2k+1}{2}}(\SS^d)$. While the operator
\begin{equation}
    \mathcal{G}:~\mu\mapsto\int_{\SS^d}\sigma_k(\theta\cdot\eta)d\mu(\theta),\qquad\mu\in\mathcal{M}(\SS^d)
\end{equation}
characterize the Barron space $\mathcal{B}^k(\SS^d)$ given by the norm (in comparison with \eqref{eqn:barron_norm_int})
\begin{equation}
    \|f\|_{\mathcal{B}^k(\SS^d)}=\inf\limits_{\mu\in\mathcal{M}(\SS^d)}\Bigl\{|\mu|(\SS^d):~f=\mathcal{G}(\mu)\Bigr\}
\end{equation}
From Remark \ref{rem:even_odd_necess}, up to a polynomial, a neuron $\sigma_k(\theta\cdot\eta)$ is essentially an even/odd function on the sphere if $k$ is odd/even. So for all the spaces on the sphere, it suffices to consider their even/odd subspace.

For notation simplicity, in this section, we consider $k$ to be fixed. We write
\begin{equation*}
    \mathcal{M}_{*}(\SS^d):=\lt\{\mu\in\mathcal{M}(\SS^d):~\mu(A)=(-1)^{k+1}\mu(-A),\quad A\subset\SS^d\rt\},
\end{equation*}
and for any function space $\mathcal{S}(\SS^d)$,
\begin{equation*}
    \mathcal{S}_{*}(\SS^d):=\lt\{f\in\mathcal{S}(\SS^d):~f(\eta)=(-1)^{k+1}f(-\eta),\quad \eta\in\SS^d\rt\}.
\end{equation*}
Namely, $\mathcal{M}_{*}(\SS^d)$ (or $\mathcal{S}_{*}(\SS^d)$) consists of even measures (or functions) if $k$ is odd and odd (even) measures (or functions) if $k$ is even.


\begin{theorem}\label{thm:sph_int_rep_sob}
For $d\in\NN$, $k\in\NN_0$, the map $\mathcal{G}:~\mathcal{M}_{*}(\SS^d)\to\mathcal{B}^{k}_{*}(\SS^d)$ is an isomorphism, and satisfies
\begin{equation}\label{eqn:mu_Gmu_equiv}
    |\mu|(\SS^d)=\|\mathcal{G}(\mu)\|_{\mathcal{B}^k(\SS^d)}.
\end{equation}
Furthermore, $\mathcal{G}: ~\mathcal{L}_{*}^2(\SS^d)\to\mathcal{H}^{\frac{d+2k+1}{2}}_*(\SS^d)$ is an isomorphism on subspaces with
\begin{equation}
    \|\psi\|_{\mathcal{L}^2(\SS^d)}\simeq\|\mathcal{G}(\psi)\|_{\mathcal{H}^{\frac{d+2k+1}{2}}(\SS^d)}.
\end{equation}

\end{theorem}

\begin{proof}
Consider the space $\mathcal{M}_{*}(\SS^d)$. By the earlier work (see, e.g., \cite[Lemma 3]{siegel2023characterization}),
$$\mathcal{G}:~\mathcal{M}_{*}(\SS^d)\to\mathcal{B}_{*}^k(\SS^d)$$
is surjective. To see it is an injection, let $\mu_1,\mu_2\in\mathcal{M}_{*}(\SS^d)$ be different. For the harmonic basis, we denote
\begin{equation*}
    \widehat{\mu_i}(m,\ell)=\int_{\SS^d}Y_{m,\ell}(\theta)d\mu_i(\theta),\qquad i=1,2,
\end{equation*}
then there exists some $(m',\ell')$ such that
\begin{equation}
    \widehat{\mu_1}(m',\ell')\neq\widehat{\mu_2}(m',\ell'),
\end{equation}
On the other hand, by Fubini's theorem, the $\mathcal{L}^2$-projections satisfy
\begin{equation*}
    \Pi_n(\mathcal{G}(\mu_i))(\eta)=\int_{\SS^d}\Pi_n(\sigma_k(\theta\cdot\eta))d\mu_i(\theta)=\sul_{m=0}^n\widehat{\sigma_k}(m)\sul_{\ell=1}^{N(m)}\widehat{\mu_i}(m,\ell)Y_{m,\ell}(\eta),\qquad i=1,2.
\end{equation*}
Therefore,
\begin{equation*}
    \begin{split}
        \|\mathcal{G}(\mu_1)-\mathcal{G}(\mu_2)\|_{\mathcal{L}^2(\SS^d)}^2=&\lim\limits_{n\to\infty}\left\|\Pi_n(\mathcal{G}(\mu_1))-\Pi_n(\mathcal{G}(\mu_2))\right\|_{\mathcal{L}^2(\SS^d)}^2\\
        \geq&\widehat{\sigma_k}(m')^2\left(\widehat{\mu_1}(m',\ell')-\widehat{\mu_2}(m',\ell')\right)^2>0.
    \end{split}
\end{equation*}
As $\mathcal{B}^k(\SS^d)$ is a subspace of $\mathcal{L}^2(\SS^d)$,
\begin{equation}
    \|\mathcal{G}(\mu_1)-\mathcal{G}(\mu_2)\|_{\mathcal{B}^k(\SS^d)}\gtrsim\|\mathcal{G}(\mu_1)-\mathcal{G}(\mu_2)\|_{\mathcal{L}^2(\SS^d)}>0.
\end{equation}
Hence $\mathcal{G}:~\mathcal{M}_{*}(\SS^d)\to\mathcal{B}^{k}_{*}(\SS^d)$ is an isomorphism. Then the definition of $\|\cdot\|_{\mathcal{B}^k(\SS^d)}$ gives
\begin{equation}
    \|f\|_{\mathcal{B}^k(\SS^d)}=\inf\limits_{\mu\in\mathcal{M}(\SS^d)}\left\{|\mu|(\SS^d):~f(x)=\int_{\SS^d}\sigma_k(\theta\cdot\eta)d\mu(\theta)\right\}=|\mathcal{G}^{-1}(\mu)|(\SS^d),\quad f\in\mathcal{B}^k_{*}(\SS^d),
\end{equation}
which is exactly \eqref{eqn:mu_Gmu_equiv}.

Now we prove the result on $\mathcal{L}^2_{*}(\SS^d)$. It suffices to apply the orthogonality and write
\begin{equation*}
    \begin{split}
        \mathcal{G}(\psi)(\eta)=&\fint_{\SS^d}\lrt{\sul_{m\equiv k+1~\mathrm{mod}~2}\sul_{\ell=1}^{N(m)}\widehat{\psi}(m,\ell)Y_{m,\ell}(\theta)}\lrt{\sul_{m'\equiv k+1~\mathrm{mod}~2}\widehat{\sigma_k}(m')\sul_{\ell=1}^{N(m')}Y_{m',\ell}(\theta)Y_{m',\ell}(\eta)}d\theta\\
        =&\sul_{m\equiv k+1~\mathrm{mod}~2}\widehat{\sigma_k}(m)\sul_{\ell=1}^{N(m)}\widehat{\psi}(m,\ell)Y_{m,\ell}(\eta).
    \end{split}
\end{equation*}
By Lemma \ref{lem:Bach},
\begin{equation*}
    (m^{d+2k+1}+1)\widehat{\sigma_k}(m)^2\simeq1,\qquad m\equiv k+1~\mathrm{mod}~2
\end{equation*}
then the norm $\|\mathcal{G}(\psi)\|_{\mathcal{H}^{\frac{d+2k+1}{2}}(\SS^d)}$ can be estimated by \eqref{eqn:Sob_norm_Parseval} as
\begin{equation*}
    \begin{split}
        \|\mathcal{G}(\psi)\|_{\mathcal{H}^{\frac{d+2k+1}{2}}(\SS^d)}^2=&\sul_{m\equiv k+1~\mathrm{mod}~2}(m^{d+2k+1}+1)\widehat{\sigma_k}(m)^2\sul_{\ell=1}^{N(m)}\widehat{\psi}(m,\ell)^2\\
        \simeq&\sul_{m\equiv k+1~\mathrm{mod}~2}\sul_{\ell=1}^{N(m)}\widehat{\psi}(m,\ell)^2=\|\psi\|_{\mathcal{L}^2(\SS^d)}^2.
    \end{split}
\end{equation*}
Conversely, denote $\mathcal{G}^{-1}$ by
\begin{equation*}
    \mathcal{G}^{-1}(\psi)=\sul_{m\equiv k+1~\mathrm{mod}~2}\widehat{\sigma_k}(m)^{-1}\sul_{\ell=1}^{N(m)}\widehat{\psi}(m,\ell)Y_{m,\ell}(\eta),\qquad \psi\in\mathcal{H}^{\frac{d+2k+1}{2}}_{*}(\SS^d),
\end{equation*}
then
\begin{equation*}
    \begin{split}
        \|\mathcal{G}^{-1}(\psi)\|_{\mathcal{L}^2(\SS^d)}^2=&\sul_{m\equiv k+1~\mathrm{mod}~2}\widehat{\sigma_k}(m)^{-2}\sul_{\ell=1}^{N(m)}\widehat{\psi}(m,\ell)^2\\
        \simeq&\sul_{m\equiv k+1~\mathrm{mod}~2}(m^{d+2k+1}+1)\sul_{\ell=1}^{N(m)}\widehat{\psi}(m,\ell)^2
        =\|\psi\|_{\mathcal{H}^{\frac{d+2k+1}{2}}(\SS^d)}^2.
    \end{split}
\end{equation*}
Finally, it is straightforward to check $\mathcal{G}^{-1}$ is the inverse of $\mathcal{G}$.

\end{proof}

With this theorem, the difference of the Barron spaces and Sobolev spaces can be further characterized as
\begin{equation}\label{eqn:chara_space_sphere}
    \begin{split}
        &\mathcal{B}_{*}^k(\SS^d)/\ \mathcal{H}^{\frac{d+2k+1}{2}}_{*}(\SS^d)\cong\mathcal{M}_{*}(\SS^d)/\ \mathcal{L}_{*}^2(\SS^d)\\
        \cong&\mathcal{M}^{\perp}_{*}(\SS^d)\oplus\lrt{\mathcal{L}^1_{*}(\SS^d)/\ \mathcal{L}^2_{*}(\SS^d)},
    \end{split}
\end{equation}
where $\mathcal{M}^{\perp}_{*}(\SS^d)$ is the space of discrete measures, and $\lrt{\mathcal{L}^1_{*}(\SS^d)/\ \mathcal{L}^2_{*}(\SS^d)}$ can be characterized by the norm
\begin{equation*}
    \|[f]\|_{\lrt{\mathcal{L}^1_{*}(\SS^d)/\ \mathcal{L}^2_{*}(\SS^d)}}=\inf\limits_{g\in\mathcal{L}^2_{*}(\SS^d)}\|f-g\|_{\mathcal{L}^1(\SS^d)}
\end{equation*}
for any coset $[f]$ represented by some $f\in\mathcal{L}^1_{*}(\SS^d)$.

\section{Proof of main results on general domains}\label{sec:proof_ball}

Building on Theorem \ref{thm:appr_rate_sph}-\ref{thm:sph_int_rep_sob}, we now establish their analogues on the ball. The proof relies on an operator, originally introduced in \cite{bach2017breaking}, that maps ReLU$^k$ functions defined on $\mathbb{S}^d$ to ReLU$^k$ functions on $\mathbb{B}^d$. We note that the existence of such an operator relies on the homogeneity of the activation function.

While there exists a more natural homeomorphism from a cap of $\mathbb{S}^d$ to the ball $\mathbb{B}^d$, the operator defined below has the distinct advantage of preserving the structure of ReLU$^k$ functions. This property is crucial for extending the spherical result to the ball, enabling us to rigorously prove Theorem \ref{thm:appr_rate_ball}-\ref{thm:barron_sob_equ} as a direct consequence of its spherical counterpart.

In the remainder of this section, we define this operator, explore its key properties, and demonstrate how it facilitates the desired approximation result on the ball.

\begin{definition}
    Given $d\in\NN$, $k\in\NN_0$, let
    \begin{equation*}
        G:=\left\{\eta\in\SS^d:~\eta_{d+1}\geq\frac{1}{\sqrt{2}}\right\}.
    \end{equation*}    
    Define operator $S_k:\ \mathcal{L}^2(G)\to \mathcal{L}^2(\BB^d)$ by
\begin{equation*}
    (S_kg)(x):=|\tilde x|^kg\left(\frac{\tilde x}{|\tilde x|}\right)
    ,\quad x\in\BB^d
\end{equation*}
and operator $T_k:\ \mathcal{L}^2(\BB^d)\to \mathcal{L}^2(G)$ by
\begin{equation*}
    T_kf(\eta)=\eta_{d+1}^kf\lrt{\frac{\bar\eta}{\eta_{d+1}}},\qquad \eta\in G.
\end{equation*}
where $\tilde x$ has been defined earlier as $\displaystyle\tilde x=\binom{x}{1}$, and $\bar\eta=(\eta_1,\dots,\eta_d)^\top$.
\end{definition}
Following the idea in \cite{bach2017breaking} and \cite{yang2024optimal}, it is straightforward to verify that
 \begin{equation}
    S_kT_k=\mathrm{id}_{\mathcal{L}^2(\BB^d)},\qquad T_kS_k=\mathrm{id}_{\mathcal{L}^2(G)}
\end{equation}
and
\begin{equation}
\begin{split}
    &(S_k\sigma_k(\theta\cdot\circ))(x)=\sigma_k(\theta\cdot\tilde x),\qquad x\in\BB^d,\\
    &(T_k\sigma_k(\theta\cdot\tilde\circ))(\eta)=\sigma_k(\theta\cdot\eta),\qquad \eta\in G.
\end{split}
\end{equation}
    \begin{lemma}\label{lem:ball_sphere}
    For any $\rr\ge 0$,
    \begin{equation}\label{eqn:equiv_Sk}
        \|g\|_{\mathcal{H}^\rr( G)}\simeq\|S_kg\|_{\mathcal{H}^\rr(\BB^d)}
    \end{equation}
 \end{lemma}
\begin{proof}
   To start up, we assume $r\in\NN$. Write $S_kg$ as
    \begin{equation}
        S_kg(x)=\zeta_1(x)g(\lambda_1(x)),\qquad x\in\BB^d,
    \end{equation}
    where $\zeta_1:\BB^d\to\RR$ and $\lambda_1:\BB^d\to G$ are given as
    $$\lambda_1(x)=\frac{\tilde x}{|\tilde x|}=\frac{1}{\sqrt{|x|^2+1}}\binom{x}{1},\quad \zeta_1(x)=|\tilde x|^k=\lrt{|x|^2+1}^{k/2},\qquad x\in\BB^d.$$
    with $\lambda_1$ has its inverse
\begin{equation}\label{eqn:lam_1_inverse}
    \lambda_2(\eta)=\frac{\bar\eta}{\eta_{d+1}}=\lrt{\frac{\eta_1}{\eta_{d+1}},\dots,\frac{\eta_d}{\eta_{d+1}}}^\top,\qquad \eta\in G.
\end{equation}
Suppose $g$ be smooth enough, then the multivariate chain rule gives
\begin{equation*}
    \begin{split}
        \lt|\frac{\partial^{|\alpha|}(S_kg)}{\partial x_1^{\alpha_1}\dots\partial x_d^{\alpha_d}}(x)\rt|\lesssim&\max\limits_{|\gamma|\leq|\alpha|}\lt|\frac{\partial^{|\gamma|} g}{\partial x_1^{\gamma_1}\dots\partial x_{d+1}^{\gamma_{d+1}}}(\lambda_1(x))\rt|\|\lambda_1\|_{\mathcal{W}^{|\alpha|,\infty}(\BB^d)}\|\zeta_1\|_{\mathcal{W}^{|\alpha|,\infty}(\BB^d)}\\
        \lesssim&\max\limits_{|\gamma|\leq|\alpha|}\lt|\frac{\partial^{|\gamma|} g}{\partial x_1^{\gamma_1}\dots\partial x_{d+1}^{\gamma_{d+1}}}(\lambda_1(x))\rt|.
    \end{split}
\end{equation*}
As the trivial projection
$$\eta\mapsto\bar\eta=(\eta_1,\dots,\eta_d)^\top$$
is a homeomorphism from $ G\to\{x\in\BB^d:~|x|\leq1/\sqrt{2}\}$, $\lambda_2$ is then also a homeomorphism from $ G$ to its image $\BB^d$. Consequently,
\begin{equation}\label{eqn:S_kf_les_f}
    \|S_kg\|_{\mathcal{H}^\rr(\BB^d)}\lesssim\|g\|_{\mathcal{H}^\rr( G)}.
\end{equation}
As \eqref{eqn:S_kf_les_f} holds for all sufficiently smooth functions $g$, it holds for all $g\in\mathcal{H}^\rr( G)$.

To prove $\|g\|_{\mathcal{H}^\rr(\Omega)}\lesssim\|S_kg\|_{\mathcal{H}^\rr(\BB^d)}$, it suffices to take $f=S_kg$ and prove
\begin{equation}\label{eqn:T_kg_les_g}
    \|T_kf\|_{\mathcal{H}^\rr(\Omega)}\lesssim\|f\|_{\mathcal{H}^\rr(\BB^d)}.
\end{equation}
We write
\begin{equation}
    T_kf(\eta)=\zeta_2(\eta)f(\lambda_2(\eta)),\qquad \eta\in G,
\end{equation}
where $\lambda_2$ as in \eqref{eqn:lam_1_inverse} and $\zeta_2(\eta)=\eta_{d+1}^k$. Then \eqref{eqn:T_kg_les_g} can be proved using the same argument as above. Thus for $r\in\NN$,
\begin{equation*}
    \|g\|_{\mathcal{H}^\rr( G)}\simeq\|S_kg\|_{\mathcal{H}^\rr(\BB^d)}
\end{equation*}

For $r\notin\NN$, it suffices to apply
\begin{equation*}
    \|g\|_{\mathcal{H}^{\lfloor r\rfloor}( G)}\simeq\|S_kg\|_{\mathcal{H}^{\lfloor r\rfloor}(\BB^d)},\qquad\|g\|_{\mathcal{H}^{\lceil r\rceil}( G)}\simeq\|S_kg\|_{\mathcal{H}^{\lceil r\rceil}(\BB^d)},
\end{equation*}
then the interpolation theory of linear operators yields
\begin{equation*}
    \|S_kg\|_{\mathcal{H}^\rr(\BB^d)}\lesssim\|g\|_{\mathcal{H}^\rr(G)},\quad\|T_kf\|_{\mathcal{H}^\rr(G)}\lesssim\|f\|_{\mathcal{H}^\rr(\BB^d)}
\end{equation*}
and completes the proof of this lemma.
\end{proof}

\subsection{Proof of Theorem \ref{thm:appr_rate_ball}}
We are now in the position to provide the detailed proof for Theorem \ref{thm:appr_rate_ball}.
\begin{proof}
Without loss of generality, assume 
$\Omega\subset\BB^d$. By the classical extension theory \cite{stein1970singular}, there exists an extension $f_E$ such that $$\|f_E\|_{\mathcal{H}^\rr(\BB^d)}\leq\|f_E\|_{\mathcal{H}^\rr(\RR^d)}\lesssim\|f\|_{\mathcal{H}^\rr(\Omega)}.$$
By Lemma \ref{lem:ball_sphere}, we have
    $$\|T_kf_E\|_{\mathcal{H}^\rr( G)}\lesssim\|f_E\|_{\mathcal{H}^\rr(\BB^d)}\lesssim\|f\|_{\mathcal{H}^\rr(\Omega)}.$$
    By the extension theorems again, there exists a function $g$ on $\SS^d$, which is an extension of $T_kf_E$, such that
    \begin{equation*}
        \begin{split}
            &\|g\|_{\mathcal{H}^\rr(\SS^d)}\lesssim\|T_kf_E\|_{\mathcal{H}^\rr( G)}\lesssim\|f\|_{\mathcal{H}^\rr(\Omega)},\\
            &g(\eta)=(-1)^{k+1}g(-\eta),\qquad \eta\in\SS^d.
        \end{split}
    \end{equation*}
    Now by Theorem \ref{thm:appr_rate_sph},  there exists a vector $a\in\RR^n$ with $\|a\|_2\lesssim h^{-\frac{2k+1-2\rr}{2}}\|f\|_{\mathcal{H}^\rr(\BB^d)}$ such that
    \begin{equation}\label{eqn:est_general_nonunif}
        \begin{split}
            &\biggl\|T_kf_E-\sul_{j=1}^na_j\sigma_k(\theta_j^*\cdot\circ)\biggr\|_{\mathcal{H}^\ss( G)}\leq\Bigl\|g-\sul_{j=1}^na_j\sigma_k(\theta_j^*\cdot\circ)\Bigr\|_{\mathcal{H}^\ss(\SS^d)}\\
            \lesssim& \lrt{\frac{h}{\underline{h}}}^{d+1}h^{\rr-\ss}\lt\|g\rt\|_{\mathcal{H}^\rr(\SS^d)}\lesssim\lrt{\frac{h}{\underline{h}}}^{d+1}h^{\rr-\ss}\lt\|f\rt\|_{\mathcal{H}^\rr(\Omega)}.
        \end{split}
    \end{equation}
    Applying Lemma \ref{lem:ball_sphere} again, we obtain
    \begin{equation}
        \begin{split}
            &\Bigl\|f_E-\sul_{j=1}^ma_j\sigma_k(\circ\cdot w_j^*+ b_j^*)\biggr\|_{\mathcal{H}^\ss(\BB^d)}=\Bigl\|S_k\bigl(T_kf_E-\sul_{j=1}^ma_j\sigma_k(\theta_j^*\cdot\circ)\bigr)\biggr\|_{\mathcal{H}^\ss(\BB^d)}\\
            \lesssim&\Bigl\|T_kf_E-\sul_{j=1}^ma_j\sigma_k(\theta_j^*\cdot\circ)\Bigr\|_{\mathcal{H}^\ss( G)}\lesssim\lrt{\frac{h}{\underline{h}}}^{d+1}h^{\rr-\ss}\lt\|f\rt\|_{\mathcal{H}^\rr(\Omega)}.
        \end{split}
    \end{equation}
    This proves 
    \begin{equation}
        \Bigl\|f-\sul_{j=1}^n a_j\phi_j\Bigr\|_{\mathcal{H}^\ss(\Omega)}\lesssim h^{\rr-\ss}\lt\|f\rt\|_{\mathcal{H}^\rr(\Omega)}.
    \end{equation}

\end{proof}

\subsection{Proof of Theorem \ref{thm:barron_sob_equ}}
\begin{proof}
    Without loss of generality, we assume $\Omega\subset\BB^d$. By the extension theorem again, for $f\in\mathcal{H}^{\frac{d+2k+1}{2}}(\Omega)$, there exists
    $g$ on $\SS^d$, which is an extension of $T_kf$, such that
    \begin{equation*}
        \begin{split}
            &\|g\|_{\mathcal{H}^{\frac{d+2k+1}{2}}(\SS^d)}\lesssim\|T_kf\|_{\mathcal{H}^{\frac{d+2k+1}{2}}( G)}\lesssim\|f\|_{\mathcal{H}^{\frac{d+2k+1}{2}}(\Omega)},\\
            &g(\eta)=(-1)^{k+1}g(-\eta),\qquad \eta\in\SS^d.
        \end{split}
    \end{equation*}
    By Theorem \ref{thm:sph_int_rep_sob}, there exists $\psi_0\in\mathcal{L}^2(\SS^d)$ such that
    \begin{equation*}
        g(\eta)=\fint_{\SS^d}\psi_0(\theta)\sigma_k(\theta\cdot \eta)d\theta.
    \end{equation*}
    Then
    \begin{equation*}
        \begin{split}
            &\fint_{\SS^d}\psi_0(\theta)\sigma_k(\theta\cdot\tilde x)d\theta=\fint_{\SS^d}\psi_0(\theta)|\tilde x|^k\sigma_k\lrt{\theta\cdot\frac{\tilde x}{|\tilde x|}}d\theta\\
            =&|\tilde x|^kg\lrt{\frac{\tilde x}{|\tilde x|}}=S_kg(x)=f(x),\qquad x\in\BB^d,
        \end{split}
    \end{equation*}
    hence Theorem \ref{thm:sph_int_rep_sob} implies
    \begin{equation*}
        \begin{split}
            \|f\|_{\mathcal{H}^{\frac{d+2k+1}{2}}(\Omega)}\gtrsim&\|g\|_{\mathcal{H}^{\frac{d+2k+1}{2}}(\SS^d)}\simeq\|\psi_0\|_{\mathcal{L}^2(\SS^d)}\\
            \geq&\inf\limits_{\psi\in\mathcal{L}^2(\SS^d)}\lt\{\|\psi\|_{\mathcal{L}^2(\SS^d)}:~\int_{\SS^d}\sigma_k(\theta\cdot\tilde x)\psi(\theta)d\theta=f(x)\rt\}.
        \end{split}
    \end{equation*}
    Now suppose there is some $\psi_1$ such that
    \begin{equation*}
        \int_{\SS^d}\sigma_k(\theta\cdot\tilde x)\psi_1(\theta)d\theta=f(x),
    \end{equation*}
    then
    \begin{equation*}
        \begin{split}
            \fint_{\SS^d}\psi_1(\theta)\sigma_k(\theta\cdot \eta)d\theta=&\fint_{\SS^d}\psi_1(\theta)\eta_{d+1}^k\sigma_k\lrt{\theta\cdot\frac{\eta}{\eta_{d+1}}}d\theta\\
            =&\eta_{d+1}^kf\lt(\frac{\bar\eta}{\eta_{d+1}}\rt)=T_kf(\eta).
        \end{split}
    \end{equation*}
    Again, Theorem \ref{thm:sph_int_rep_sob} implies
    \begin{equation*}
        \|f\|_{\mathcal{H}^{\frac{d+2k+1}{2}}(\Omega)}\lesssim\|f\|_{\mathcal{H}^{\frac{d+2k+1}{2}}(\BB^d)}\simeq\|T_kf\|_{\mathcal{H}^{\frac{d+2k+1}{2}}(\SS^d)}\lesssim\|\psi_1\|_{\mathcal{L}^2(\SS^d)},
    \end{equation*}
    which yields
    \begin{equation*}
        \|f\|_{\mathcal{H}^{\frac{d+2k+1}{2}}(\Omega)}\lesssim\inf\limits_{\psi\in\mathcal{L}^2(\SS^d)}\lt\{\|\psi\|_{\mathcal{L}^2(\SS^d)}:~\int_{\SS^d}\sigma_k(\theta\cdot\tilde x)\psi(\theta)d\theta=f(x)\rt\}.
    \end{equation*}
\end{proof}

\section{Deterministic analysis of randomized neural networks:\\ questioning the necessity of randomization}\label{sec:random}

Neural network training involves non-convex optimization, which is computationally challenging. While greedy methods \cite{siegel2023greedy, xu2024efficient} offer convergence guarantees, they can be costly. Randomized approaches, such as stochastic basis selection \cite{Igelnik1995}, Extreme Learning Machines (ELM) \cite{Huang2006, huang2006universal, Liu2014}, and random features \cite{rahimi2007random, rahimi2008weighted, li2019towards, bach2017equivalence, nelsen2021random, mei2022generalization}, provide efficient alternatives by fixing hidden layer parameters randomly (often uniformly or Gaussian \cite{he2015delving}) and only optimizing linear weights. These methods have shown empirical success in various tasks, including function approximation, classification, and solving PDEs \cite{lopez2014randomized, belkin2020two, gerace2020generalisation, hu2022universality, dwivedi2020physics, dong2021local, chen2022bridging, zhang2024transferable, dang2024local, chi2024random}. However, the theoretical justification for specific sampling strategies often lags behind empirical performance.

Let $\mathbb{E}_{n,\mu}$ denote expectation over $n$ i.i.d. samples $\{\theta_j\}_{j=1}^n$ from a probability distribution $\mu$ on $\SS^d$. If $\mu$ is uniform, we write $\mathbb{E}_{n}$. Standard concentration inequalities applied to the integral representation \eqref{eqn:f_expect_repre} yield the follwoing theorem.
  
\begin{theorem}\label{thm:rand_half}
    Let $d,k,n,\Omega$ be as in Theorem \ref{thm:appr_rate_ball}, then for $f\in\mathcal{H}^{\frac{d+2k+1}{2}}(\Omega)$,
    \begin{equation}
        \mathbb{E}_n\Bigl[\inf\limits_{a\in\RR^n}\Bigl\|f-\sul_{j=1}^na_j\phi_j \Bigr\|^2_{\mathcal{L}^2(\Omega)}\Bigr]\lesssim n^{-1}\|f\|^2_{\mathcal{H}^{\frac{d+2k+1}{2}}(\Omega)}.
    \end{equation} 
\end{theorem}
This $\mathcal{O}(n^{-1/2})$ $\mathcal{L}^2$-rate provides theoretical backing for random feature methods. (Note: Theorem \ref{thm:barron_sob_equ} allows deterministic points + uniform sampling, improving on \cite{siegel2022sharp} where the measure was unknown. We focus here on the stronger results derived from Theorem \ref{thm:appr_rate_ball}.)

Leveraging our deterministic approximation result (Theorem \ref{thm:appr_rate_ball}), we can refine the analysis of random sampling:

\begin{theorem}\label{thm:random_linear}
Let $\{\theta_j\}_{j=1}^n$ be i.i.d. uniform samples from $\SS^d$. With probability at least $1-\delta$, there exists $a\in\RR^n$ with $\|a\|_2\lesssim \lrt{\frac{n}{\log(n/\delta)}}^{\frac{2k+1-2\rr}{2d}}\|f\|_{\mathcal{H}^{\rr}(\Omega)}$ such that
    \begin{equation}\label{eqn:rate_random}
        \Bigl\|f-\sul_{j=1}^n a_j\phi_j\Bigr\|_{\mathcal{H}^\ss(\Omega)}\lesssim\lrt{\frac{n}{\log(n/\delta)}}^{-\frac{\rr-\ss}{d}}\lt\|f\rt\|_{\mathcal{H}^\rr(\Omega)}.
    \end{equation}
    Consequently, the expected error satisfies
    \begin{equation}\label{eqn:rate_expectation}
        \mathbb{E}_n\Bigl[\inf\limits_{a\in\RR^n}\Bigl\|f-\sul_{j=1}^n a_j\phi_j\Bigr\|_{\mathcal{H}^\ss(\Omega)}\Bigr]\lesssim \lrt{\frac{n}{\log n}}^{-\frac{\rr-\ss}{d}} \|f\|_{\mathcal{H}^\rr(\Omega)}.
    \end{equation}
    In particular, for $\mathcal{L}^2$ approximation ($\ss=0$, $\rr=\frac{d+2k+1}{2}$),
        \begin{equation}
        \mathbb{E}_n\Bigl[\inf\limits_{a\in\RR^n}\Bigl\|f-\sul_{j=1}^n a_j\phi_j\Bigr\|_{\mathcal{L}^2(\Omega)}\Bigr]\lesssim\lrt{\frac{n}{\log n}}^{-\frac{1}{2}-\frac{2k+1} {2d}}\|f\|_{\mathcal{H}^{\frac{d+2k+1}{2}}(\Omega)}.
    \end{equation}
\end{theorem}

\begin{proof}
The proof relies on Theorem \ref{thm:appr_rate_ball}, which depends on $h=\max\limits_{\theta\in\SS^d}\min\limits_{1\leq j\leq n}\rho(\theta,\theta_j)$. Standard covering arguments and concentration inequalities show that for i.i.d. uniform points $\{\theta_j\}$, $h \lesssim (n/\log(n/\delta))^{-1/d}$ with probability at least $1-\delta$. Substituting this into \eqref{eqn:main_res_nonuniform} yields the result.
\end{proof}

This analysis clarifies why randomization works: it efficiently generates parameter distributions $\{\theta_j\}$ that are nearly well-distributed (i.e., $h$ is small) with high probability. However, comparing Theorem \ref{thm:random_linear} with Theorem \ref{thm:appr_rate_ball} reveals that deterministically chosen well-distributed points achieve a *better* approximation rate (lacking the $\log n$ factor) and eliminate the small probability $\delta$ of failure associated with random sampling. Since well-distributed point sets on spheres can be constructed deterministically, we conclude that they are sufficient and theoretically preferable to random sampling for achieving optimal approximation rates in this context. Randomization is thus a practical means to approximate a desirable deterministic configuration.

\section{Generalization Analysis for FNS using Rademacher Complexity} \label{sec:gene_analy}

We analyze the generalization error when approximating a target function $f$ using a dataset $D_m$. Given the approximation rate \eqref{eqn:main_Omega_beta} and coefficient bound $M$ from Theorem \ref{thm:appr_rate_ball}, we can estimate the expected error between $f$ and an approximant constructed from $m$ i.i.d. samples.

Standard generalization analysis uses concentration inequalities \cite{boucheron2003concentration} and complexity measures like Rademacher complexity, covering numbers, or spectral complexity \cite{cucker2007learning, shalev2014understanding}. These bounds depend critically on the coefficient norms of the approximants.

We illustrate using a second-order linear elliptic PDE with Neumann boundary conditions on a bounded $C^\infty$ domain $\Omega \subset \mathbb{R}^d$ with $|\Omega|=1$:
\begin{equation}\label{eq:pde_strong}
\begin{cases}
    -\Delta g(x) + g(x) = h(x), & x \in \Omega \\
     \frac{\partial g}{\partial n}(x)  = 0, & x \in \partial \Omega,
\end{cases}
\end{equation}
where $h\in\mathcal{L}^\infty(\Omega)$ is given. The variational formulation seeks $g \in \mathcal{H}^1(\Omega)$ minimizing the energy functional:
\begin{equation}\label{eqn:exp_loss_rad_fns_revised}
    \mathcal{E}(g) := \int_\Omega \Psi(g)(x) \, dx, \quad \text{where } \Psi(g)(x) = \frac{1}{2} |\nabla g(x)|^2 + \frac{1}{2} g(x)^2 - h(x)g(x).
\end{equation}
The target function $f$ is the unique minimizer of $\mathcal{E}(g)$ due to strict convexity:
\begin{equation}\label{eqn:true_minimizer}
    f := \arg\min_{g \in \mathcal{H}^1(\Omega)} \mathcal{E}(g).
\end{equation}

Using Monte Carlo, we draw $m$ i.i.d. uniform samples $x_1, \dots, x_m$ from $\Omega$ and approximate $\mathcal{E}(g)$ with the empirical risk:
\begin{equation}\label{eqn:emp_risk_rad_fns_revised}
    \mathcal{E}_m(g) := \frac{1}{m} \sum_{i=1}^m \Psi(g)(x_i).
\end{equation}
Assuming $f \in \mathcal{H}^{\frac{d+2k+1}{2}}(\Omega)$, we use the hypothesis space $L_{n,M}^k$ (from Theorem \ref{thm:appr_rate_ball}). The empirical risk minimizer (ERM) is:
\begin{equation}\label{eqn:emp_minimizer_rad_fns_revised}
    f_{n,m} := \arg\min\limits_{g \in L_{n,M}^k} \mathcal{E}_m(g).
\end{equation}
Since $L_{n,M}^k$ is linear and $\mathcal{E}_m(g)$ is quadratic in coefficients, $f_{n,m}$ is unique and efficiently computable via convex optimization.

The generalization error is the expected excess risk:
\begin{equation}\label{eqn:general_error}
    \mathbb{E}_{x_1,\dots,x_m}\Big[\mathcal{E}(f_{n,m}) - \mathcal{E}(f)\Big].
\end{equation}
Due to strong convexity of $\mathcal{E}(g)$, bounding this also bounds the expected squared $\mathcal{H}^1(\Omega)$ error $\mathbb{E}_{x_1,\dots,x_m}\Big[\|f_{n,m} - f\|_{\mathcal{H}^1(\Omega)}^2\Big]$.

\subsection{Rademacher Complexity}
We use Rademacher complexity \cite{bartlett2002rademacher} to bound \eqref{eqn:general_error}. For a function class $\mathcal{F}:\Omega\to\RR$, its Rademacher complexity is:
\begin{equation}\label{eqn:Rademacher}
    R_m(\mathcal{F})=\mathbb{E}_{x_i, \xi_i}\Big[\sup\limits_{g\in\mathcal{F}}\frac{1}{m}\sul_{i=1}^m\xi_ig(x_i)\Big],
\end{equation}
where $x_i$ are i.i.d. uniform samples and $\xi_i$ are i.i.d. Rademacher variables ($\pm 1$).

We bound $R_m(\mathcal{F}_{n,M})$ for $\mathcal{F}_{n,M} := \{\Psi(g) : g \in L_{n,M}^k\}$. Since $L_{n,M}^k \subset \Sigma_{n,M}^k$, monotonicity gives $R_m(\mathcal{F}_{n,M}) \le R_m(\{\Psi(g) : g \in \Sigma_{n,M}^k\})$.
Using standard Rademacher complexity properties (sum rule, Lipschitz composition, product rule \cite{hong2021priori, bartlett2002rademacher, mohri2018foundations, wainwright2019high}) and known bounds $R_m(\Sigma_{n,M}^k) \lesssim M m^{-1/2}$ and $R_m(\{\partial g/\partial x_j : g \in \Sigma_{n,M}^k\}) \lesssim M m^{-1/2}$ \cite[Theorem 6]{hong2021priori}, we get:
\begin{equation}\label{eqn:Rademacher_comp}
    R_m(\mathcal{F}_{n,M}) \le R_m\Big(\Big\{\Psi(g):~g\in\Sigma_{n,M}^k\Big\}\Big) \lesssim (M+\|h\|_{\mathcal{L}^\infty(\Omega)})Mm^{-\frac{1}{2}}.
\end{equation}

\subsection{Generalization analysis}
Combining Theorem \ref{thm:appr_rate_ball} and the complexity bound \eqref{eqn:Rademacher_comp} yields the generalization error estimate.

\begin{theorem}\label{thm:generalization}
    Let $\Omega\subset \mathbb{R}^d$ be a bounded $C^\infty$ domain. Let $f$, $f_{n,m}$ be defined by \eqref{eqn:true_minimizer} and \eqref{eqn:emp_minimizer_rad_fns_revised}. If $f\in\mathcal{H}^{\frac{d+2k+1}{2}}(\Omega)$, the expected excess risk satisfies:
\begin{equation}\label{eqn:gener_main_error}
    \mathbb{E}_{x_1,\dots,x_m}\Big[\mathcal{E}(f_{n,m}) - \mathcal{E}(f)\Big]\lesssim (M+\|h\|_{\mathcal{L}^\infty(\Omega)})Mm^{-\frac{1}{2}}+M^2n^{-1-\frac{2k-1}{d}}.
\end{equation}
    Choosing $n=\lceil m^{\frac{d}{2(d+2k-1)}}\rceil$ balances the terms, giving:
\begin{equation}
    \mathbb{E}_{x_1,\dots,x_m}\Big[\mathcal{E}(f_{n,m}) - \mathcal{E}(f)\Big]\lesssim (M+\|h\|_{\mathcal{L}^\infty(\Omega)})Mm^{-\frac{1}{2}}.
\end{equation}
By strong convexity of $\mathcal{E}(g)$, the expected squared $\mathcal H^1(\Omega)$ error is bounded:
\begin{equation}\label{eqn:gener_H1}
    \mathbb{E}_{x_1,\dots,x_m}\Big[\|f_{n,m} - f\|_{\mathcal H^1(\Omega)}^2\Big]\lesssim (M+\|h\|_{\mathcal{L}^\infty(\Omega)})Mm^{-\frac{1}{2}}.
\end{equation}
\end{theorem}

The proof follows standard arguments \cite{bartlett2002rademacher, mohri2018foundations}. The first term in \eqref{eqn:gener_main_error} arises from the Rademacher complexity bound \eqref{eqn:Rademacher_comp} via generalization bounds. The second term is the approximation error from Theorem \ref{thm:appr_rate_ball}. Strong convexity yields \eqref{eqn:gener_H1}. See \cite{hong2021priori} for similar detailed proofs for PDE approximation.

\section{Concluding remarks}
In this paper, we developed a new integral representation of Sobolev space using ReLU$^k$ function, and compared it with the corresponding result of Barron spaces. We also showed in Sobolev spaces, the approximation properties of nonlinear shallow neural networks can be fully realized through simple linearization, and provided an upper bound of the corresponding coefficients. Such an upper bound is essential in approximation theory, and allows the corresponding generalization analysis. It is also worth noting that the techniques in this paper also inspires a Bernstein inequality for ReLU$k$ neural networks, for which we are completing a new paper entitled “Bernstein Inequalities for Linearized ReLU$k$ Neural Networks and Applications”.

Since a DNN is essentially a composition of shallow neural networks, we hope that the findings in this paper provide valuable insights into the role of nonlinearity within deep neural networks. Furthermore, we aim for this work to inspire further mathematical research into the interplay between nonlinearity and expressive power (approximation properties) in deep learning.

\bibliographystyle{plain}
\bibliography{ref}

\end{document}